\documentclass[a4paper,11pt]{article}

\usepackage{amsmath,amsthm,amstext,amscd,amssymb,euscript,mathrsfs}
\usepackage{calrsfs}
\usepackage{epsfig}

\usepackage{color}
\usepackage{esint}
\usepackage{pdfsync}
\usepackage{graphicx}
\usepackage{todonotes}

\newcommand{\Z}{\mathbb Z}
\newcommand{\R}{\mathbb R}

\newcommand{\E}{\mathbb E}

\renewcommand{\phi}{\varphi}

\def\1{{\mathchoice {\rm 1\mskip-4mu l} {\rm 1\mskip-4mu l}
{\rm 1\mskip-4.5mu l} {\rm 1\mskip-5mu l}}}

\newtheorem{theorem}{{\small T}{\scriptsize HEOREM}}[section]
\newtheorem{corollary}{{\bf{\small C}{\scriptsize OROLLARY}}}[section]
\newtheorem{proposition}{{\bf{\small P}{\scriptsize ROPOSITION}}}[section]
\newtheorem{lemma}{{\bf{\small L}{\scriptsize EMMA}}}[section]
\newtheorem{remark}{{\bf{\small R}{\scriptsize EMARK}}}[section]
\newtheorem{definition}{{\bf{\small D}{\scriptsize EFINITION}}}[section]

\renewenvironment{proof}[1]
{\noindent{{\bf{\small{P}{\scriptsize ROOF}}}.}\hspace{0.1cm} #1} {$\;\qed$\newline}

\newcommand{\beq}{\begin{eqnarray}}
\newcommand{\eeq}{\end{eqnarray}}

\newcommand{\ba}{\begin{align*}}
\newcommand{\ea}{\end{align*}}

\newcommand{\be}{\begin{equation}}
\newcommand{\ee}{\end{equation}}

\newcommand{\bl}{\begin{lemma}}
\newcommand{\el}{\end{lemma}}

\newcommand{\br}{\begin{remark}}
\newcommand{\er}{\end{remark}}

\newcommand{\bt}{\begin{theorem}}
\newcommand{\et}{\end{theorem}}

\newcommand{\bd}{\begin{definition}}
\newcommand{\ed}{\end{definition}}

\newcommand{\bp}{\begin{proposition}}
\newcommand{\ep}{\end{proposition}}

\newcommand{\bc}{\begin{corollary}}
\newcommand{\ec}{\end{corollary}}

\newcommand{\bpr}{\begin{proof}}
\newcommand{\epr}{\end{proof}}

\newcommand{\bi}{\begin{itemize}}
\newcommand{\ei}{\end{itemize}}

\newcommand{\ben}{\begin{enumerate}}
\newcommand{\een}{\end{enumerate}}

\newcommand{\caL}{{\mathcal L}}

\renewcommand{\(}{\left(}        \renewcommand{\)}{\right)}
\renewcommand{\[}{\left[}        \renewcommand{\]}{\right]}

     \newcommand{\nn}{\nonumber}

\newcommand{\red}[1]{\textcolor[rgb]{0,0,0}{#1}}

\title{A generalized Asymmetric  Exclusion Process\\
 \red{with $U_q(\mathfrak{sl}_2)$}  stochastic duality}

\date{\today}

\begin{document}

\author{
Gioia Carinci$^{\textup{{\tiny(a)}}}$,
Cristian Giardin{\`a}$^{\textup{{\tiny(a)}}}$,\\
Frank Redig$^{\textup{{\tiny(b)}}}$,
Tomohiro Sasamoto $^{\textup{{\tiny(c)}}}$.\\\\
{\small $^{\textup{(a)}}$ Department of Mathematics, University of Modena and Reggio Emilia}\\
{\small via G. Campi 213/b, 41125 Modena, Italy}
\\
{\small $^{\textup{(b)}}$ Delft Institute of Applied Mathematics, Technische Universiteit Delft}\\
{\small Mekelweg 4, 2628 CD Delft, The Netherlands}\\
{\small $^{\textup{(c)}}$ Department of Physics, Tokyo Institute of Technology,}\\
{\small 2-12-1 Ookayama, Meguro-ku, Tokyo, 152-8550, Japan}\\
}

\maketitle

\begin{abstract}
We {study} a new process, which we call ASEP$(q,j)$, where particles move asymmetrically on a one-dimensional
integer lattice with a bias determined by $q\in (0,1)$ and where at most $2j\in\mathbb{N}$ particles per site are allowed.
The process is constructed from a $(2j+1)$-dimensional representation of a quantum Hamiltonian with 
\red{$U_q(\mathfrak{sl}_2)$} invariance by applying a suitable
ground-state transformation. After showing basic properties of 
the process ASEP$(q,j)$, we prove self-duality
with several self-duality functions constructed from
the symmetries of the quantum Hamiltonian. 
By making use of the self-duality property 
we compute the first $q$-exponential moment of the current
for step initial conditions (both a shock or a rarefaction fan)
as well as when the process is started from an homogeneous 
product measure.
\end{abstract}


\section{Introduction}

\subsection{Motivation}
The {\em Asymmetric Simple Exclusion Process} (ASEP) on $\mathbb{Z}$ is one of the most popular
interacting particle system. For each $q\in(0,1]$, the process is
defined, up to an irrelevant time-scale factor, by the following two rules:
i) each site is vacant or occupied; ii) particles sitting at occupied sites
try to jump at rate $q$ to the left and at rate $q^{-1}$ to the right  and they succeed if the arrival
site is empty. The ASEP plays a crucial role in the development of
the mathematical theory of non-equilibrium statistical mechanics, similar to the role of Ising model
for equilibrium statistical mechanics. However, whereas the Ising model -- defined for dichotomic
spin variables -- is easily generalizable to variables taking more than two values  (Potts model),
there are a-priori  different possibilities to define the ASEP process with more than one particle
per site and it is not clear what  the best option is.

In the analysis of the standard (i.e., maximum one particle per site) Exclusion Process a very
important property of the model is played by {\em self-duality}. First established in the
context of the Symmetric Simple Exclusion Process (SSEP) \cite{Liggett}, self-duality is a key
tool that allows to study the process using only a finite number of dual particles.  For instance, using
self-duality and coupling techniques Spitzer and Liggett were able to show that the only extreme translation
invariant measures for the SSEP on $\mathbb{Z}^d$ are the Bernoulli product measures and to identify
the domain of attraction of them.
The extension of duality to ASEP is due to Sch\"{u}tz \cite{Schutz} and has played an important role in
showing that ASEP is included in the KPZ universality class, see e.g. \cite{borodin, corwin}.
As a general rule, the extension of a duality relation from a symmetric to an asymmetric process is far
from trivial.

{\em It is the aim of this paper to provide a {generalization} of the ASEP with multiple
occupation per site for which (self-)duality can be established}.  A guiding principle in the search of
such process will be the connection between Exclusion Processes and Quantum Spin Chains.
{The duality property will then be used to study the statistics of the current of particles
for the process on the infinite lattice.}

\subsection{Previous extensions of the ASEP}


Several extensions of the ASEP
model allowing multiple occupancy at each site have been provided
and studied in the literature. Among them we mention the following.
\begin{itemize}
\item[a)]
It is well known that the XXX Heisenberg quantum spin chain with spin $j=1/2$
is related (by a change of basis) to the SSEP. In this mapping the spins
are represented by $2\times 2$ matrices satisfying the \red{$\mathfrak{sl}_2$} algebra.
By considering higher values of the spins, represented by $(2j+1)$-dimensional
matrices with $j\in\mathbb{N}/2$,
{one  obtains the {\em generalized Symmetric Simple
Exclusion Process with up to $2j$ particles per site} 
(SSEP($2j$) for short), sometimes also called ``partial exclusion'' \cite{caput, SS,GKRV}}.
Namely, denoting by $\eta_i\in\{0,1,\ldots 2j\}$ the number of
particles at site $i\in\mathbb{Z}$, the process that is obtained
has rates $\eta_i(2j-\eta_{i+1})$ for a particle jump from site $i$ to site $i+1$
and rate $\eta_{i+1}(2j-\eta_{i})$ for the reversed jump.
For such extension of the SSEP, duality can be formulated and
(extreme) translation invariant measures are provided by the Binomial product
measures with parameters $2j$ (the number of trials) and $\rho$ (the success
probability in each trial).

The naive asymmetric version of this process, i.e., considering
a rate $q\, n_{i+1}(2j-n_{i})$ for the jump of a particle from site $i+1$ to site $i$
and a rate $q^{-1} n_i(2j-n_{i+1})$ for the jump of a particle from site $i$ to site $i+1$,
with $q\in(0,1)$, loses the \red{$\mathfrak{sl}_2$} symmetry and has no other symmetries from which
duality functions can be obtained. In fact in this model, there is no 
self-duality property
expect in the case $j=1/2$ where it coincides with ASEP \cite{Schutz}.
\item[b)]
Another possibility is to consider  the so-called {\em $K$-exclusion
process }\cite{K} that simply
gives rates $1$ to particle jumps from occupied sites
together with the exclusion rule that prevents more than
$K$ particles to accumulate on each site ($K\in\mathbb{N}$).
Namely, denoting by $\mathbf{1}_A$ the indicator function of the set $A$, 
the $K$-exclusion process on $\mathbb{Z}$ has rates
$\mathbf{1}_{\{\eta_i > 1, \, \eta_{i+1} < K\}}$ for the jump from
site $i$ to site $i+1$ and $\mathbf{1}_{\{\eta_{i+1} > 1, \, \eta_{i} < K\}}$
for the jump from site $i+1$ to site $i$.
For the symmetric version of this process it has been shown
in \cite{K} that extremal translation invariant measures are product
 measures (with truncated-geometric marginals).
The asymmetric version of this process obtained by giving
rate $q$ to (say) the left jumps and rate $q^{-1}$
for the right jumps, has been studied by Sepp\"{a}l\"{a}inen
(see \cite{seppalainen} and references therein).
For the asymmetric process, invariant measures are unknown, and non-product,
nevertheless many properties of this process (e.g. hydrodynamic limit) could be established.
For this process, both in the symmetric and asymmetric case, there is no duality.
\end{itemize}

\subsection{Informal description of the results}

The fact that self-duality is known for the Symmetric  Exclusion Process for
any $j\in\mathbb{N}/2$ \cite{GKRV} and it is unknown in all the other cases
(except ASEP with $j=1/2$) can be traced back to the link that it exists 
between self-duality and the algebraic structure of interacting 
particle systems.
Such underlying structure is usually provided by a {\em Lie algebra} naturally 
associated to the generator of the  process.
The first result in this direction was given in \cite{SS} for the
{\em symmetric} process,
while a systematic and general approach has been described in \cite{GKRV}, \cite{CGGR}.
When passing from symmetric to asymmetric processes, 
one has 
{to change} from the original Lie algebra to the corresponding
{\em deformed quantum Lie algebra}, where the deformation parameter is related to the asymmetry.
This  was noticed in \cite{Schutz} for the standard ASEP, which corresponds
to a representation of the \red{$U_q(\mathfrak{sl}_2)$} algebra with spin $j=1/2$.

In this paper we fully unveil the relation between the deformed \red{$U_q(\mathfrak{sl}_2)$} algebra
and a suitable generalization of the Asymmetric  Simple Exclusion Process.
{\em For a given $q\in(0,1)$ and $j\in \mathbb{N}/2$, we construct a new process, that we name
ASEP$(q,j)$}, which provides an extension of the standard ASEP process to
a situation where sites can accommodate more than one (namely $2j$) particles.
The construction is based on a quantum Hamiltonian \cite{B}, which up to a constant can be obtained from
the Casimir operator and a suitable co-product structure
of the quantum Lie algebra \red{$U_q(\mathfrak{sl}_2)$}. For this Hamiltonian we construct a ground-state which is a tensor product
over lattice sites. This ground-state is used to transform the Hamiltonian into the generator of the
Markov process ASEP$(q,j)$ via
a ground-state transformation.
As a result of the symmetries of the Hamiltonian, we obtain several self-duality functions of
the associated ASEP$(q,j)$. Those functions are then used in the study
of the statistics of the current that flows through the system for different
initial conditions.

For $j=1/2$ the ASEP$(q,j)$ reduces to the standard ASEP.
{For $j \to \infty$, after a proper time-rescaling, ASEP$(q,j)$ converges to {the so-called $q$-TAZRP 
(Totally Asymmetric Zero Range process)}, see Remark \ref{TAZRP} below and \cite{borodin} for more datails.}

{We mention  also \cite{NSS} and \cite{M} for  
other processes 
with \red{$U_q(\mathfrak{sl}_2)$} symmetry.  
In particular the process in \cite{M}  is  a 
$(2j+1)$
state partial exclusion process 
constructed using the Temperley-Lieb algebra, in which multiple jumps of particles between neighboring sites
are allowed. We remark that for $j=1$ the process depends on a parameter $\beta$ and
for the special value  $\beta=0$ it reduces to ASEP$(q,1)$.}
 

\subsection{From quantum Lie algebras to self-dual Markov processes}

By analyzing in full details the case of the \red{$U_q(\mathfrak{sl}_2)$}
we will elucidate a general scheme that can be applied to other
algebras, thus providing asymmetric version of other 
interacting particle systems 
(e.g. independent random walkers, zero-range process, inclusion process).
We highlight below the main steps of the scheme
(at the end of each step we point to the section where
such step is made for  \red{$U_q(\mathfrak{sl}_2)$}).
\begin{enumerate}
\item[i)]
({\em Quantum Lie Algebra}):
Start from the quantization $U_q(\mathfrak{g})$ of the enveloping algebra
$U(\mathfrak{g})$ of a Lie algebra $\mathfrak{g}$
(Sect. \ref{quantum-algebra}).
\item[ii)]
({\em Co-product}):
Consider a co-product $\Delta: U_q(\mathfrak{g}) \to U_q(\mathfrak{g})\otimes U_q(\mathfrak{g}) $
making the quantized enveloping algebra a bialgebra
(Sect. \ref{cooooo}).
\item[iii)]
({\em Quantum Hamiltonian}):
For a given representation of the quantum Lie algebra $U_q(\mathfrak{g})$
compute the co-product $\Delta(C)$ of a Casimir element $C$
(or an element in the centre of the algebra).
{For a one-dimensional chain of size $L$ construct the quantum Hamiltonian $H_{(L)}$ by
summing up copies of $\Delta(C)$ over nearest neighbor edges.}
(Sect. \ref{q-h}).
\item[iv)]
({\em Symmetries}): Basic symmetries {(i.e. commuting operators)} of the quantum Hamiltonian are constructed
by applying the co-product to the generators of the quantum Lie algebra.
(Sect. \ref{basic}).
\item[v)]
({\em Ground state transformation}): 
Apply a ground state transformation to the quantum Hamiltonian
$H_{(L)}$ to turn it into the generator ${\cal L}^{(L)}$ of a Markov stochastic 
process 
(Sect. \ref{proc}).
\item[vi)]
({\em Self-duality}):
Self-duality functions of the Markov process are obtained by 
acting with (a function) of the basic symmetries on the reversible measure
of the process.
(Sect. \ref{D}).
\end{enumerate}
Whereas steps i)--iv) depend on the specific choice of the
quantum Lie algebra, the last two steps are independent
of the particular choice but require additional hypotheses. 
In particular whether step v) is possible depends on the specific
properties of the Hamiltonian and its ground state.
They are further discussed in Section \ref{2}.

\subsection{Organization of the paper}
The rest of the paper is organized as follows.
In Section \ref{2} we give the general strategy to construct a self-dual Markov process from a {quantum} Hamiltonian,
a positive ground state and a symmetry. {In the case where the quantum Hamiltonian is  given by a finite dimensional matrix the strategy actually amounts to a similarity transformation with the diagonal matrix constructed from the
ground state components}. 

In Section \ref{3} we start by defining the ASEP$(q,j)$ process. 
After proving some of its basic properties in theorem \ref{basicproptheorem}
(e.g. existence of non-homogenous product measure and absence of
translation invariant product measure), we enunciate our main results. 
They include: the self-duality property of the (finite or infinite) ASEP$(q,j)$
(theorem \ref{main}) and its use in the computation of some exponential 
moments of the total integrated current via a single dual asymmetric walker
(lemma \ref{Lemma:N}).
The explicit computation are shown for the step initial conditions 
(theorem \ref{current-step}) and when the process
is started from an homogenous product measure
(theorem \ref{current-prod}).

The remaining Sections contain the algebraic construction of the  
ASEP$(q,j)$ process by the implementation of the steps described 
in the above scheme for the case of the quantum Lie algebra
$U_q(\mathfrak{g})$. 
In particular, in Section \ref{4} we introduce the quantum Hamiltonian
and its basic symmetries on which we base our construction 
of the ASEP$(q,j)$. 
In Section \ref{proc} we exhibit a non trivial $q$-exponential 
symmetry and a positive ground state of the quantum Hamiltonian
that allows to define a Markov process. 
In Section \ref{D} we prove the main self-duality result for the ASEP$(q,j)$. 
In Section \ref{C} we explore other choices for the symmetries of the Hamiltonian
and, as a consequence, prove the existence of an alternative duality function 
that  reduces to the Sch\"{u}tz duality function for the case $j=1/2$.

\section{Ground state transformation and self-duality}
\label{2}

In this section we describe a general strategy to construct a Markov process from
a {quantum} Hamiltonian. {Furthermore we illustrate how to derive} self-duality 
functions for that Markov process from symmetries
of the Hamiltonian. The construction of a Markov process from a Hamiltonian and
a positive ground state 
has been used at several places, e.g. the Ornstein-Uhlenbeck process
is constructed in this way from the harmonic oscillator Hamiltonian, see e.g. \cite{simon}.
In lemma \ref{abstractlemma} below we recall the procedure, and how to recover symmetries of the
Markov process from symmetries of the Hamiltonian.
In general this procedure to be applied requires some condition
on the Hamiltonian. In the discrete setting this condition
boils down to non-negative out-of-diagonal elements
and the existence of a positive ground state.
In the more general setting the Hamiltonian has to be
a Markov generator up to mass conservation 
(cfr. \eqref{mass}).

\subsection{Ground state transformation and symmetries}
\begin{lemma}\label{abstractlemma}
Let $h$ be a bounded continuous function and let $L$ be the generator
of a Markov process on a metric space $\Omega$.
Let $A$ be an operator of the form
\begin{equation}
\label{mass}
A f= Lf -hf
\end{equation}
Suppose that there exists $\psi$ such that $e^\psi$ is in the domain of $A$, and
\begin{equation}\label{assum}
Ae^\psi=0.
\end{equation}
{Then the following holds:}
\begin{itemize}
\item[a)] The operator defined by
\begin{equation}
L_\psi f= e^{-\psi} A(e^\psi f)
\end{equation}
is a Markov generator.
\item[b)] There is a one-to-one correspondence between symmetries
(commuting operators) of $A$ and
symmetries of $L_\psi$:
$[S,A]=SA-AS=0$ if and only if $[L_\psi, S_\psi]=0$ with
$S_\psi= e^{-\psi} S e^\psi$.
\item[c)] If $A$ is self-adjoint on {the space} $L^2(\Omega, d\alpha)$ for some
$\sigma$-finite measure $\alpha$ on $\Omega$, then
$L_\psi$ is self-adjoint on $L^2(\Omega, d\mu)$ with $\mu(dx)= e^{2\psi(x)}\alpha(dx)$.
In particular, if $\int e^{2\psi(x)}\alpha(dx)=1$
then $\mu$ is a reversible probability measure for the Markov process with
generator $L_\psi$.
\end{itemize}
\end{lemma}
\begin{proof}
For item a): for every $\phi$ such that $e^\phi$ is in the domain of $L$, the operator
\begin{equation}
\label{l-phi}
L_\phi f= e^{-\phi} L(e^{\phi} f)- (e^{-\phi} L(e^{\phi})) f
\end{equation}
defines a Markov generator, see e.g. \cite{fk} section 1.2.2, and \cite{pr}.
Now choosing $\phi=\psi$, we obtain from the assumption
\eqref{assum} that
$$
e^{-\psi} Le^\psi= h
$$
Hence,
\begin{eqnarray*}
L_\psi f
&=& e^{-\psi} L(e^{\psi} f)- (e^{-\psi} L(e^{\psi})) f
\\
&=& e^{-\psi} L(e^{\psi} f)- h f= e^{-\psi}( L-h)( e^\psi f)
\\
&=& e^{-\psi} A(e^\psi f)
\end{eqnarray*}
For item b) suppose that $S$ commutes with $A$, then
\begin{eqnarray*}
L_\psi S_\psi &=& e^{-\psi} A e^\psi e^{-\psi} S e^\psi
\\
&=&
e^{-\psi} AS e^\psi=  e^{-\psi} SA e^\psi
\\
&=& S_\psi L_\psi
\end{eqnarray*}
For item c), we compute
\begin{eqnarray*}
\int g L_\psi (f) d\mu &=&  \int g (e^{-\psi} A(e^\psi f)) e^{2\psi} d\alpha
\\
&=&
\int e^\psi g A(e^\psi f) d\alpha
\\
&=&
\int A(e^\psi g) (e^\psi f) d\alpha = \int (L_\psi g) f d\mu
\end{eqnarray*}
where in the third equality we used $A=A^*$ in $L^2(\Omega, d\alpha)$.
\end{proof}
The following is a restatement of lemma \ref{abstractlemma} in the context of a finite state space
{$\Omega$ with cardinality $|\Omega|<\infty$. In this case}  the
condition $A=L-{h}
$ just means that $A$ has non-negative off diagonal elements.
\bc\label{corooo}
Let $A$ be a {$|\Omega|\times |\Omega|$} matrix with non-negative off diagonal elements.
Suppose there exists a column vector 
$e^{\psi}:= g \in \R^{{|\Omega|}}$ 
 with strictly positive entries and such that
$Ag=0$. Let us denote by $G$ the diagonal matrix with entries 
{$G(x,x)=g(x)$ for $x\in\Omega$}.
Then we have the following
\begin{itemize}
\item[a)] The matrix
$$
\caL= G^{-1} A G
$$
with entries
{
\be\label{amodi}
\caL(x,y) = \frac{A(x,y) g(y)}{g(x)}, \qquad x,y \in \Omega\times\Omega
\ee
}
is the generator of a Markov process $\{X_t:t\geq 0\}$ taking values on {$\Omega$}.
\item[b)] $S$ commutes with $A$ if and only if $G^{-1} S G$ commutes with $\caL$.
\item[c)] If $A=A^*$, where $^*$ denotes transposition, then
the probability measure {$\mu$} on {$\Omega$}
{
\be\label{revmes}
\mu(x)= \frac{(g(x))^2}{\sum_{x\in\Omega} (g(x))^2}
\ee
}
is reversible for the process with generator $\caL$.

\end{itemize}
\ec

\begin{proof}
The proof of the corollary is obtained by specializing the statements of the  lemma \ref{abstractlemma} to the
finite dimensional setting.  In particular for item a), the operator $L_{\phi}$ in \eqref{l-phi} reads
$$
(L_{\phi} f)(x) = \sum_{y\in\Omega} L(x,y)e^{\phi(y)-\phi(x)} (f(y)-f(x))\;.
$$
Putting $\phi(x) = \psi(x)$ and using  the condition $\sum_{y\in\Omega} L(x,y) e^{\psi(y)} = h(x) e^{\psi(x)}$ 
one finds
$$
(L_{\psi} f)(x) = \sum_{y\in\Omega} A(x,y)e^{\psi(y)-\psi(x)} f(y)\
$$
from which \eqref{amodi} follows.
\end{proof}
\br\label{nontrivial}
Notice that for every column vector $f$ we have that if $Af=0$ then for any $S$ commuting with $A$ (symmetry of $A$) we have
$ASf=SAf=0$. This will be useful later on (see section \ref{sec5.3}) when starting from a vector $f$ with some entries equal to zero, we can
produce, by acting with a symmetry $S$, a vector
$g= Sh$
which has all entries strictly positive.
\er
\subsection{Self-duality and symmetries}
For the discussion of self-duality, we restrict to the case of a finite state space $\Omega$.
\bd[{Self-duality}]
We say that a Markov process  ${X}:=\{X_t:t\geq 0\}$ on $\Omega$ is self-dual with self-duality function
$D:\Omega\times\Omega\to\R$ if for all $x,y\in \Omega$ {and for all $t>0$}
\be\label{self}
\E_x D(X_t, y) = \E_y D(x, Y_t)\;.
\ee
{Here $\E_x(\cdot)$ denotes expectation with respect to the process $X$ initialed at $x$ at time $t=0$
and $Y$ denotes a copy of the process started at $y$.}
\ed
\noindent
This is equivalent to its infinitesimal reformulation, i.e., if the Markov process
$X$
has generator $\caL$ 
then \eqref{self} holds if and only if
\be\label{LD=DLT}
\caL D= D\caL^*
\ee
where $D$ is the 
{$|\Omega|\times|\Omega|$ matrix} with entries $D(x,y)$ for  $x,y\in \Omega$.
We recall two general results on self-duality from \cite{GKRV}.
\begin{itemize}
\item[a)] {\em {Trivial} duality function from a reversible measure.}

If the process $\{X_t:t\geq 0\}$ has a reversible measure $\mu(x)>0$, then
by the detailed balance condition, it is easy to check that the diagonal matrix
\be \label{Cheap}
D(x,y)= \frac{1}{\mu(x)} \delta_{x,y}
\ee
is a self-duality function.
\item[b)] {\em New duality functions via symmetries.}

If $D$ is a self-duality function and $S$ is a symmetry of $\caL$, then
$SD$ is a self-duality function.
\end{itemize}
We can then combine corollary \ref{corooo}
with these results to obtain the following.
\bp\label{dualprop}
Let $A=A^*$ be a matrix with non-negative off-diagonal elements, and $g$
an eigenvector of $A$ with eigenvalue zero, with strictly positive entries. Let $\caL= {G^{-1}AG}$ be the corresponding
Markov generator.
Let $S$ be a symmetry of $A$, then $G^{-1}SG^{-1}$  is a self-duality function for the process with generator $\caL$.
\ep
\bpr
By item c) of the corollary \ref{corooo} combined with item a) of the general facts on self-duality
we conclude that $G^{-2}$ is a self-duality function.
By item b) of corollary \ref{corooo} we conclude that if $S$ is a symmetry of
$A$ then $G^{-1} SG$ is a symmetry of $\caL$.
Then, using item b)
of the general facts on self-duality
we conclude that $G^{-1} S G G^{-2}= G^{-1} S G^{-1}$ is a self-duality function for the process
with generator $\caL$.
\epr
\section{The asymmetric exclusion process with parameters $(q,j)$ (ASEP$(q,j)$)}
\label{3}

\noindent
{\bf{\small{N}{\scriptsize OTATION}}}.
For $q\in(0,1)$  and $n\in \mathbb{N}_0$ we introduce the {$q$-number}
\begin{equation}
\label{q-num}
[n]_q=\frac{q^n-q^{-n}}{q-q^{-1}}
\end{equation}
satisfying the property $\lim_{q\to1} [n]_q = n$.
The first $q$-number's are thus given by
$$[
0]_q = 0, \quad\quad\quad [1]_q=1, \quad \quad \quad [2]_q=q+ q^{-1}, \quad \quad \quad [3]_q=q^2+1+q^{-2}, \quad \dots
$$
We also introduce the $q$-factorial
$$
[n]_q!:=[n]_q \cdot [n-1]_q \cdot \dots \cdot [1]_q \;,
$$
and the $q$-binomial coefficient
$$
\binom{n}{k}_q:=\frac{[n]_q!}{[k]_q! [n-k]_q!}\;.
$$

\subsection{Process definition}

\noindent
We start with the definition of 
{a novel} interacting
particle systems.

\bd[ASEP$(q,j)$ process]
Let $q\in(0,1)$ and  $j\in\mathbb{N}/2$.
{For a given vertex set $V$, denote by  $\eta= (\eta_i)_{i\in V}$  a particle configuration
belonging to the state space $\{0,1,\ldots,2j\}^V$ so that $\eta_i$ is interpreted as the number
of particles at site $i\in V$. Let $\eta^{i,k}$ denotes the
particle configuration that is obtained from $\eta$ by moving a particle
from site $i$ to site $k$.}
\begin{itemize}
\item[a)]
 The Markov process ASEP$(q,j)$ on  
 $[1,L]\cap\mathbb{Z}$ with 
{closed} boundary conditions
is defined by the generator
\begin{eqnarray}
\label{gen}
&&({\cal L}^{(L)}f)(\eta) =\sum_{i=1}^{{L-1}} ({\cal L}_{i,i+1}f)(\eta) \qquad \text{with} \nonumber\\
({\cal L}_{i,i+1}f)(\eta) 
& = & q^{\eta_i-\eta_{i+1}-(2j+1)} [\eta_i]_q [2j-\eta_{i+1}]_q (f(\eta^{i,i+1}) - f(\eta)) \nonumber \\
& +  & q^{\eta_i-\eta_{i+1}+(2j+1)} [2j -\eta_i]_q [\eta_{i+1}]_q (f(\eta^{i+1,i}) - f(\eta))
\end{eqnarray}
\item[b)]
We call the infinite-volume  ASEP$(q,j)$ on $\mathbb{Z}$ the process whose generator is given by
\be\label{genZ}
({\cal L}^{(\Z)}f)(\eta) =\sum_{i\in \mathbb Z} ({\cal L}_{i,i+1}f)(\eta)
\ee
\item[c)]
The ASEP$(q,j)$ on the
torus $\mathbf T_L:=\mathbb{Z}/L \mathbb Z$ with periodic boundary conditions
is defined as the Markov process with generator
\be\label{genT}
({\cal L}^{(T_L)}f)(\eta) =\sum_{i\in \mathbf T_L} ({\cal L}_{i,i+1}f)(\eta)
\ee
\end{itemize}
\vspace{-2cm}
\begin{figure}[h]
\centering
\includegraphics[width=13cm]{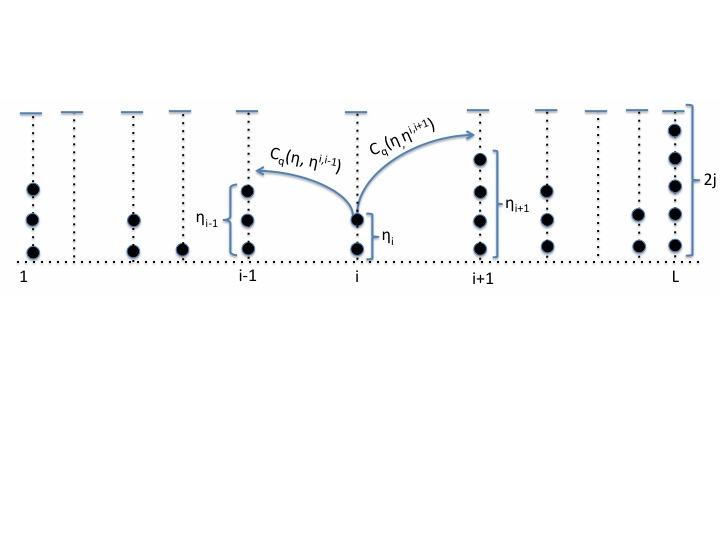}
\vspace{-4cm}
\caption{\label{figSIP} Schematic description of the ASEP($(q,j)$). The arrows represent the possible transitions and the corresponding rates $c_q(\eta,\xi)$ are given in \eqref{rates} below. Each site can accomodate at most $2j$ particles.}
\end{figure}

\be\label{rates}
c_q(\eta,\xi)= 
\left\{
\begin{array}{ll}
q^{\eta_i-\eta_{i+1}-(2j+1)} [\eta_i]_q [2j-\eta_{i+1}]_q & \text{if }\: \xi=\eta^{i,i+1}\\
 q^{\eta_{i-1}-\eta_{i}+(2j+1)} [2j -\eta_{i-1}]_q [\eta_{i}]_q & \text{if }\:\xi=\eta^{i,i-1}\\
0 & \text{otherwise}
\end{array}
\right. 
\ee
\ed

\br[The standard ASEP]
In the case $j=1/2$ each site can accommodate at most one particle and the ASEP$(q,j)$ reduces to the standard ASEP
with jump rate to the left equal to $q$ and jump rate to the right equal to $q^{-1}$.
\er

\br[The symmetric process]
In the {limit $q\to1$} the ASEP$(q,j)$ reduces to the SSEP$(2j)$, i.e. the generalized
simple symmetric exclusion process with up to $2j$ particles per site
(also called partial exclusion) (see \cite{caput, SS,GKRV,GRV}). 
All the results of the present paper apply also to this symmetric case.
In particular, for $q\to 1$, the duality functions that will be given in theorem
\ref{main} below reduce to the duality functions of the SSEP. 
\er

\br[Connection with the $q$-TAZRP]\label{TAZRP} Consider the process $y^{(j)}_t:=\{y_i^{(j)}(t)\}_{i \in \mathbb Z}$ obtained from the ASEP$(q,j)$ after the time scale transfomation $t \to (1-q^2)q^{4j-1}t$ (i.e. $y_i^{(j)}(t) :=\eta_i((1-q^2)q^{4j-1}t)$) then, in the limit $j\to \infty$, $y^{(j)}_t$ converges to the $q$-TAZRP (Totally Asymmetric Zero Range process) in  $\mathbb Z$ whose generator is given by:
\be\label{genT2}
({\cal L}^{(q-\text{TAZRP})}f)(y) =\sum_{i\in \mathbb Z}\,\frac{1-q^{2y_i}}{1-q^2} \;[f(y^{i,i+1})-f(y)], \qquad f: \mathbb N^{\mathbb Z} \to \mathbb R
\ee
see e.g. \cite{borodin} for more details on {this} process.
\er

\subsection{Basic properties of the ASEP$(q,j)$}
We summarize basic properties of the ASEP$(q,j)$ in the following theorem.
We recall that a function $f$ is said to be monotonous if $f(\eta)\le f(\eta')$ 
whenever $\eta\le \eta'$ (in the sense of partial order) and a Markov process 
with semigroup $S(t)$ is said to be monotonous if, for every time $t\ge0$,
$S(t)f$ is  monotonous function if $f$ is a monotonous function.
In this paper we do not investigate the consequence of monotonicity
which is for instance very useful for the hydrodynamic limit
(see \cite{BGRS}). 

\begin{theorem}[{Properties of ASEP$(q,j)$ process}]
\label{basicproptheorem}
\noindent
\begin{itemize}
\item[a)] For all $L\in\mathbb{N}$, the ASEP$(q,j)$ on $[1,L]\cap\mathbb{Z}$ with closed boundary conditions 
admits a family (labeled by $\alpha >0$) of reversible product measures with marginals given by
\be
\label{stat-meas}
\mathbb{P^{(\alpha)}}(\eta_i = n) = \frac{\alpha^n}{Z^{(\alpha)}_i} \,{\binom{2j}{n}_q} \cdot  q^{2n(1+j- 2 j i)}  \qquad\qquad n=0,1,\ldots , 2j
\ee
for $i \in \{1, \ldots, L\}$ and
\be \label{Z}
Z_i^{{(\alpha)}} = \sum_{n=0}^{2j}  {\binom{2j}{n}_q} \cdot  \alpha^n q^{2n(1+j- 2 j i)}
\ee

\item[b)] The infinite volume ASEP$(q,j)$ is well-defined and admits the reversible product measures with marginals given by \eqref{stat-meas}-\eqref{Z}.

\item[c)] Both the ASEP$(q,j)$ on $[1,L]\cap\mathbb{Z}$ with closed boundary conditions and its infinite volume version are monotone processes.

\item[d)]  For $L\ge 3$, the ASEP$(q,j)$ on the Torus $\mathbf T_L$ with periodic boundary conditions does not have translation invariant stationary product measures for $j\not= 1/2$.
\item[e)]  The infinite volume ASEP$(q,j)$ 
does not have translation invariant stationary product measures for $j\not= 1/2$.
\end{itemize}
\end{theorem}
\br
Notice that of course we could have absorbed the factor {$q^{2(1+j)}$} into $\alpha$ in \eqref{stat-meas}.
However in remark \ref{natural} below we will see that the case $\alpha=1$ exactly corresponds to a natural 
ground state.
\er
\bpr

\begin{itemize}

\item[a)] Let $\mu$ be a reversible measure, then,   from detailed balance we have
\be\label{DB}
\mu(\eta)c_q(\eta, \eta^{i,i+1})=\mu(\eta^{i,i+1})c_q(\eta^{i,i+1},\eta)
\ee
where $c_q(\eta,\xi)$ are the hopping rates from $\eta$ to $\xi$ given in \eqref{rates}.
Suppose now that $\mu$ is a product measure of the form $\mu=\otimes_{i=1}^L\mu_i$ then  \eqref{DB} holds if and only if
\be
\mu_i(\eta_i-1)\mu_{i+1}(\eta_{i+1}+1) q^{2j} [2j-\eta_i+1]_q [\eta_{i+1}+1]_q = \mu_i(\eta_i)\mu_{i+1}(\eta_{i+1}) q^{-2j} [\eta_i]_q [2j-\eta_{i+1}]_q 
\ee
which implies that there exists $\beta \in \mathbb R$ so that for all $i=1, \ldots, L$
\be \label{mu}
\frac{\mu_i(n)}{\mu_i(n-1)}= \beta q^{-4ji} \frac{[2j-n+1]_q}{[n]_q}
\ee
then \eqref{stat-meas}  follows from \eqref{mu} after using an induction argument on $n$ and choosing $\beta=\alpha q^{2(j+1)}$. 
\item[b)] {The fact that the process is  well-defined} follows from standard existence criteria of \cite{Liggett}, chapter 1, while the proof of the statement on the reversible product measure is the same as in item a).

\item[c)] This follows from the fact that the rate to go from $\eta$ to $\eta^{i,i+1}$ is of
the form $b(\eta_i, \eta_{i+1})$ where $k,l\mapsto b(k,l)$ is increasing in $k$ and decreasing in $l$, and the same holds
for the rate to go from $\eta$ to $\eta^{i,i-1}$, and the general results in \cite{coc}.
\item[d)] We will prove the absence of homogeneous product measures for $j=1$, the proof for larger $j$ is similar.
Suppose that there exists an homogeneous  stationary product measure ${\bf {\bar \mu}}(\eta)=\prod_{i=1}^L \mu(\eta_i)$, then, for any function $f:\{0, \ldots, 2j\}^{\mathbb Z} \to \mathbb R$
\be\label{zero}
0=\sum_{\eta} [{\cal L}^{(T_L)}f](\eta)\bar \mu (\eta)=\sum_{\eta}f(\eta) [{\cal L}^{(T_L)*}\bar \mu] (\eta)
\ee
where
\be \label{one}
[{\cal L}^{(T_L)*}\bar \mu] (\eta)= \sum_{i\in \mathbf{T}_L} F(\eta_i, \eta_{i+1}) \bar \mu(\eta)
\ee
with
\begin{eqnarray}\label{tre}
F(\xi_1,\xi_2)&=& q^{\xi_1-\xi_{2}-2j+1}[\xi_1+1]_q [2j-\xi_{2}+1]_q \, \frac{\mu(\xi_1+1) \mu(\xi_{2}-1)}{\mu(\xi_1)\mu(\xi_{2})}\nonumber\\
&+& q^{\xi_1-\xi_{2}+2j-1}[\xi_{2}+1]_q [2j-\xi_{1}+1]_q \,  \frac{\mu(\xi_{2}+1) \mu(\xi_{1}-1)}{\mu(\xi_1)\mu(\xi_{2})}\nonumber\\
&-&  q^{\xi_1-\xi_{2}}\left(q^{-(2j+1)}+q^{2j+1}\right)  \,[\xi_1]_q [2j-\xi_{2}]_q  \label{M222}
\end{eqnarray}
Then, from \eqref{zero} and \eqref{one} we have that $\bar \mu$ is an homogeneous product measure if and only if, for all $f$,
\be\label{two}
\sum_{\eta}f(\eta) \bar \mu(\eta) \left(\sum_{i\in\mathbf{T}_L} F(\eta_i, \eta_{i+1})\right)=0
\ee
which is true if and only if
\be\label{Sum}
G(\eta):=\sum_{i\in\mathbf{T}_L} F(\eta_i, \eta_{i+1}) \equiv 0
\ee
Let $\Delta_i$ be the discrete derivative with respect to the $i$-th coordinate, i.e. let $f:\{0,\ldots, 2j\}^N\to \mathbb R$, for some $N \in \mathbb N$, then $\Delta_i f(n):= f(n+ \delta_i)-f(n)$, $n= (n_1, \ldots, n_N)$.
From  \eqref{Sum} it follows that, for any $i\in \{1, \ldots, L\}$,
\be
0=\Delta_i G(\eta)=\Delta_2 F(\eta_{i-1},\eta_i)+\Delta_1 F(\eta_{i},\eta_{i+1}) \qquad \text{for any }\quad \eta_{i-1}, \eta_i, \eta_{i+1}
\ee
this implies in particular that $\Delta_2 F(\xi_1,\xi_2)$ does not depend on $\xi_1$ and that $\Delta_1F(\xi_1,\xi_2)$ does not depend on $\xi_2$. {Therefore, necessarily} $F(\xi_1,\xi_2)$ is of the form
\be\label{FF}
F(\xi_1,\xi_2)=g(\xi_1)+h(\xi_2)
\ee
for some functions $g,h:\{0, \ldots, 2j\} \to \mathbb R$. {By using again \eqref{Sum} it follows in particular  that $F(\xi_1,\xi_1)=0$,  then, from this fact and \eqref{FF} we deduce that $h(\xi_1)=-g(\xi_1)$.}
As a consequence \eqref{Sum} holds if and only if there exists a function $g$ as above such that, for each $i \in {\mathbf T}_L$,
\begin{eqnarray}\label{teleee}
F(\eta_i,\eta_{i+1})=g(\eta_i)-g(\eta_{i+1})
\end{eqnarray}
(the opposite implication following from the fact that the sum $\left(\sum_{i\in\mathbf{T}_L} F(\eta_i, \eta_{i+1})\right)$ is now telescopic and hence zero because of periodicity).

We are going to prove now that \eqref{teleee} cannot hold for the function $F$ given in  \eqref{tre}.
Denote by
\be\label{alphaa}
\gamma:=\frac{\mu(1)^2}{\mu(2)\mu(0)} \qquad \text{and} \qquad \alpha:= q^3+q-q^{-1}-q^{-3}\;,
\ee
 fix $i$ and define $\bar \eta:=(\eta_i,\eta_{i+1})$; then, for $j=1$ the expression in \eqref{M222} becomes
\begin{eqnarray}
&&\alpha (\mathbf 1_{\bar \eta=(1,0)}- \mathbf 1_{\bar \eta=(0,1)} )+ \alpha (\mathbf 1_{\bar \eta=(2,1)}-\mathbf 1_{\bar \eta=(1,2)})\nonumber \\
&+&\left[\gamma q^3 - q -2q^{-1}-q^{-3}\right]\mathbf 1_{\bar \eta=(2,0)} - \left[q^3+2q+q^{-1}-\gamma q^{-3}\right]\mathbf 1_{\bar \eta=(0,2)}\nonumber \\
&+&\left[\gamma^{-1}(q^3+3q+3q^{-1}+q^{-3})-q^3-q^{-3}\right] \mathbf 1_{\bar \eta=(1,1)} \nonumber \\
&=&g(\eta_i)-g(\eta_{i+1}) \label{M1}
\end{eqnarray}
The condition \eqref{M1} for $\bar \eta=(1,1)$ yields that the coefficient in front of
$\mathbf 1_{\bar \eta=(1,1)}$ has to be zero, which gives
\begin{equation}\label{gammaa}
\gamma=\frac{q^3+3q+3q^{-1}+q^{-3}}{q^3+q^{-3}}
\end{equation}
with this choice of $\gamma$ \eqref{M1} gives
\begin{eqnarray}
&&\alpha (\mathbf 1_{\bar \eta=(1,0)}- \mathbf 1_{\bar \eta=(0,1)} )+ \alpha (\mathbf 1_{\bar \eta=(2,1)}-\mathbf 1_{\bar \eta=(1,2)})+\delta (\mathbf 1_{\bar \eta=(2,0)} - \mathbf 1_{\bar \eta=(0,2)})\nonumber \\
&=&g(\eta_i)-g(\eta_{i+1}) \label{M2}
\end{eqnarray}
with
\be\label{deltaa}
\delta:=\gamma q^3 -q -2q^{-1}-q^{-3}.
\ee
This yields $g(1)-g(0)=g(2)-g(1)=\alpha, g(2)-g(0)=\delta$ from which we
conclude $\delta=2\alpha$ which is in contradiction with \eqref{alphaa}, \eqref{gammaa} and \eqref{deltaa}.
\item[e)] The proof {is analogous to the proof  of item d), but it requires an extra limiting argument.} 
Namely, we want to show that the assumption of the existence
of a translation invariant product measure $\bar{\mu}$ implies 
that $\int {\cal L}^{(\Z)} f d\bar{\mu}=0$ for
every local function $f$. This leads to
$$
\sum_{i\in\Z}\int f(\eta)   F(\eta_i, \eta_{i+1}) d\bar\mu(\eta)=0
$$
for every local function $f$ and where $F(\eta_i, \eta_{i+1})$ is defined in \eqref{tre}. 
In the same spirit of point d), the proof in \cite{saada} implies  that $F(\eta_i, \eta_{i+1})$ has to be of the form
$g(\eta_i)-g(\eta_{i+1})$ which leads to the same contradiction as in item d).
\end{itemize}\epr

\subsection{Self-duality properties of the ASEP$(q,j)$}
The following self-duality theorem, together with the subsequent corollary, is the main result of the paper.
\bt[Self-duality of the finite ASEP$(q,j)$]
\label{main}
The ASEP$(q,j)$ on $[1,L]\cap\mathbb{Z}$ with closed boundary conditions
is self-dual with the following self-duality functions
\be
\label{dualll}
D_{(L)}( \eta, \xi )=
\prod_{i =1}^L  \frac{\binom{\eta_i}{\xi_i}_q}{\binom{2j}{\xi_i}_q}\,
\cdot \,
q^{(\eta_i-\xi_i)\left[2\sum_{k=1}^{i-1}\xi_k +\xi_i\right]+ 4 j i \xi_i}
\cdot \mathbf 1_{\xi_i \le \eta_i}
\ee
and
\begin{equation}\label{duality22}
 D'_{(L)}(\eta,\xi)= \prod_{i=1}^L \frac{\binom{\eta_i}{\xi_i}_q}{\binom{2j}{\xi_i}_q} \, \cdot \, q^{(\eta_i-\xi_i)\left[2\sum_{k=1}^{i-1}\eta_k -\eta_i\right]+ 4 j i \xi_i} \cdot \mathbf 1_{\xi_i \le \eta_i}
\end{equation}
\et
\bc[{Self-duality of the infinite ASEP$(q,j)$}]
\label{main2}
The ASEP$(q,j)$ on  $\mathbb Z$  is self-dual with the following self-duality functions
\be
\label{dualllxx}
D( \eta, \xi )=
\prod_{i \in \mathbb{Z}}  \frac{\binom{\eta_i}{\xi_i}_q}{\binom{2j}{\xi_i}_q}\,
\cdot \,
q^{(\eta_i-\xi_i)\left[2\sum_{k=1}^{i-1}\xi_k +\xi_i\right]+ 4 j i \xi_i}
\cdot \mathbf 1_{\xi_i \le \eta_i}
\ee
and
\begin{equation}\label{duality22xx}
 D'(\eta,\xi)= \prod_{i\in \mathbb{Z}} \frac{\binom{\eta_i}{\xi_i}_q}{\binom{2j}{\xi_i}_q} \, \cdot \, q^{(\eta_i-\xi_i)\left[2\sum_{k=1}^{i-1}\eta_k -\eta_i\right]+ 4 j i \xi_i} \cdot \mathbf 1_{\xi_i \le \eta_i}
\end{equation}
where the configurations $\eta$ and $\xi$ are such that the exponents in \eqref{dualllxx} and \eqref{duality22xx} are finite.
\ec

\noindent
The following rewriting of the duality function in \eqref{dualllxx} will be useful 
in the analysis of the current statistics.
\br For $l\in\mathbb{N}$, let $\xi^{(i_1, \ldots, i_\ell)}$ be the configurations such that
\be
\label{simple}
\xi^{(i_1, \ldots, i_\ell)}_m=
\left\{
\begin{array}{ll}
1 & \text{if }  \: m\in \{i_1, \ldots, i_\ell\}\\
0 & \text{otherwise.}
\end{array}
\right.
\ee
Define
\begin{equation}
\label{current}
 N_i(\eta):= \sum_{k\ge i} \eta_{{k}}\,,
\end{equation}
then
\begin{equation}
\label{one-dual}
D(\eta,\xi^{(i)})  = \frac{q^{4ji-1}}{q^{2j}-q^{-2j}} \, \cdot (q^{2 N_i(\eta)}-q^{2 N_{i+1}(\eta)})
\end{equation}
and more generally
$$
D(\eta,\xi^{(i_1, \ldots, i_\ell)}) = \frac{q^{4j\sum_{k=1}^\ell i_k-\ell^2}}{(q^{2j}-q^{-2j})^\ell} \, \cdot \prod_{k=1}^\ell (q^{2 N_{i_k}(\eta)}-q^{2 N_{i_k+1}(\eta)})
$$
\er

\subsection{Computation of the first $q$-exponential moment of the current for the infinite volume ASEP$(q,j)$}

We start by defining the current for the ASEP$(q,j)$ process on $\mathbb{Z}$.
\bd[Current]
The total integrated current $J_i(t)$ in the time interval $[0,t]$ is defined as the net number of particles crossing the bond $(i-1,i)$ 
in the right direction. Namely, let $(t_i)_{i\in\mathbb{N}}$ be sequence of the process jump times. Then
\be
J_i(t) = \sum_{k: t_k \in [0,t]}  (\mathbf{1}_{\{\eta(t_k) = \eta(t_k^-)^{i-1,i}\}} - \mathbf{1}_{\{\eta(t_k) = \eta(t_k^-)^{i,i-1}\}})
\ee
\ed
\bl[Current q-exponential moment via a dual walker]
\label{Lemma:N}
The total integrated current of a trajectory $(\eta(s))_{0\le s\le t}$ is given by
\be \label{J}
J_i(t):= N_i(\eta(t))-N_i(\eta(0))
\ee 
where $N_i(\eta)$ is defined in \eqref{current}.
The first $q$-exponential moment of the current when the process is started from a configuration
$\eta$ at time $t=0$ is given by
\be
\mathbb{E}_{\eta}\[q^{2J_i(t)}\] = q^{2(N(\eta)-N_i(\eta))} - \sum_{k=-\infty}^{i-1} q^{-4jk} \; \mathbf E_k \left[ q^{4jx(t)}\(1-q^{-2\eta_{x(t)}}\)\, q^{2(N_{x(t)}(\eta)-N_i(\eta))}\right]
\label{Quii}
\ee
where $N(\eta):=\sum_{i\in \mathbb Z} \eta_i$ denotes the total number
of particle (that is conserved by the dynamics),
$x(t)$ denotes a continuous time asymmetric random walker on $\mathbb Z$ 
jumping left at rate $q^{2j}[2j]_q $ and jumping right at rate $q^{-2j}[2j]_q $ 
and $\mathbf E_k$ denotes the expectation 
with respect to the law of $x(t)$ started at site $k\in\mathbb{Z}$ at time $t=0$.
Furthermore $N(\eta)-N_i(\eta) = \sum_{k<i} \eta_k$ and the first term
on the right hand side of \eqref{Quii} is zero when there are infinitely many particles
to the left of $i\in\mathbb{Z}$ in the configuration $\eta$.
\el

\bpr
\eqref{J} immediately follows from the definition of $J_i(t)$. To prove \eqref{Quii}
we start from the duality relation  
\be
\label{hello}
\mathbb{E}_{\eta}\left[D(\eta(t),\xi^{(i)})\right] = \mathbb{E}_{\xi^{(i)}}\left[D(\eta,\xi^{(x(t))})\right]
\ee
where $\xi^{(i)}$  is  the configuration with a single dual particle at site $i$
(cfr. \eqref{simple}).
Since the ASEP$(q,j)$ is self-dual the dynamics of the single dual
particle is given an asymmetric random walk $x(t)$  whose
rates are computed from the process definition and coincides
with those in the statement of the lemma.
By \eqref{one-dual} the left-hand side of \eqref{hello} is equal to
$$
\mathbb{E}_{\eta}\left[D(\eta(t),\xi^{(i)})\right] = \frac{q^{4ji-1}}{q^{2j}-q^{-2j}} \; \mathbb{E}_{\eta}\left[q^{2 N_i(t)}-q^{2 N_{i+1}(t)}\right]
$$
whereas
the right-hand side gives
$$
\mathbb{E}_{\xi^{(i)}}\left[D(\eta,\xi^{(x(t))})\right]= \frac{q^{-1}}{q^{2j}-q^{-2j}} \; \mathbf{E}_{i}\left[q^{4jx(t)}(q^{2 N_{x(t)}(\eta)}-q^{2N_{x(t)+1}(\eta)})\right]
$$
As a consequence,  for any $i \in \mathbb Z$
\begin{equation}\label{R}
\mathbb{E}_{\eta}\left[q^{2 N_i(\eta(t))}\]=\mathbb{E}_{\eta}\left[q^{2 N_{i+1}(\eta(t))}\right]+q^{-4ji}\; \mathbf{E}_{i}\left[q^{4jx(t)}(q^{2 N_{x(t)}(\eta)}-q^{2N_{x(t)+1}(\eta)})\right]
\end{equation}

\noindent
In the case of the infinite-volume ASEP$(q,j)$ the duality relation \eqref{R}   is  significant only for configurations such that $N_i(\eta(t))$ is finite for all $t$. For this reason it is convenient to divide both sides of \eqref{R} by $q^{2N_i(\eta)}$ in order to obtain a recursive relation for the current. 
Then we get from \eqref{J}
\begin{eqnarray}\label{R1}
\mathbb{E}_{\eta}\left[q^{2 J_i(t)}\]
& = &
q^{-2\eta_i}\;\mathbb{E}_{\eta}\left[q^{2 J_{i+1}(t)}\right] \nonumber \\
& + &
q^{-4ji}\; \mathbf{E}_{i}\left[q^{4jx(t)}(q^{2 (N_{x(t)}(\eta)-N_i(\eta))}-q^{2(N_{x(t)+1}(\eta)-N_i(\eta))})\right]
\end{eqnarray}
Notice that both $J_i(t)$ and $N_{x(t)}(\eta)-N_i(\eta)$ are finite quantities, for all $i$ and $t$.
By iterating the relation in \eqref{R1} and using the fact that 
$\lim_{i\to -\infty} N_i(\eta(t))=N(\eta(t))=N(\eta)$ we obtain \eqref{Quii}.
\epr

\noindent
Notice that all the quantities in \eqref{Quii} are finite for finite $t$, since $N(\eta)-N_i(\eta)>0$ and $q\le 1$.

\subsection{Step initial condition}

\bt[$q$-moment for step initial condition]
\label{current-step}
Consider the  step configurations $\eta^\pm \in\{0,\ldots, 2j\}^{\mathbb Z}$ defined as follows
\be 
\eta^+_i:=
\left\{
\begin{array}{ll}
0 &  \text{for } \; i <0\\
2j & \text{for } \;  i \ge 0
\end{array}
\right.
\qquad \qquad
\eta^-_i:=
\left\{
\begin{array}{ll}
2j &  \text{for } \; i <0\\
0 & \text{for } \;  i \ge 0
\end{array}
\right.
\ee
then, for the infinite volume ASEP$(q,j)$ we have
\be
\mathbb{E}_{\eta^+}\[q^{2J_i(t)}\]= q^{4j \max\{0,i\}} 
\left\{ 1 + q^{-4ji} \, \mathbf E_i \[\(1-q^{4jx(t)}\) \mathbf 1_{x(t)\ge 1}\]\right\}
\label{1-exp-mom-2}
\ee
and
\be
\mathbb{E}_{\eta^-}\[q^{2J_i(t)}\]= q^{-4j \max\{0,i\}} \, 
\left\{1-
\mathbf E_i \[\(1-q^{4jx(t)}\) \mathbf 1_{x(t)\ge 1}\]
\right\}
\label{1-exp-mom-4}
\ee
In the formulas above $x(t)$ denotes the random walk of Lemma 
\ref{Lemma:N} and
$$
\mathbf{E}_i(f(x(t)) = \sum_{x\in\mathbb{Z}} f(x) \cdot \mathbf{P}_i(x(t)=x)
$$
with
\begin{eqnarray}
\label{Bessel}
\mathbf{P}_i(x(t) = x) 
& = & 
\mathbb{P}(x(t)=x \;|\; x(0) = i) \nonumber \\
& = & e^{-[4j]_q t} q^{-2j(x-i)} I_{x-i}(2 [2j]_q t)
\end{eqnarray}
and $I_n(t)$ denotes the modified Bessel function.
\et
\begin{proof}
We prove only \eqref{1-exp-mom-2} since the proof of \eqref{1-exp-mom-4} is analogous. From the definition of $\eta^+$ and \eqref{Quii}, 
we have 
\begin{equation*}
\mathbb{E}_{\eta^+}\[q^{2J_i(t)}\] = q^{2(N(\eta^+)-N_i(\eta^+))} -  (1-q^{-4j})\sum_{k=-\infty}^{i-1} q^{-4jk}\; \sum_{x\ge 0}  q^{4jx}\, q^{2(N_{x}(\eta^+)-N_i(\eta^+))}\, \mathbf P_k\(x(t)=x\)
\end{equation*}
where $N(\eta^+)-N^i(\eta^+)=2j \max\{0,i\}$ and  $N_x(\eta^+)-N_i(\eta^+)=2j(\max\{0,i\}-x)$ for any $x \ge 0$. 
Then we have
\begin{eqnarray*}
\mathbb{E}_{\eta^+}\[q^{2J_i(t)}\] = q^{4j\max\{0,i\}} \left\{1+(q^{-4j}-1)F_i(t)\right\}
\end{eqnarray*}
with 
\begin{eqnarray}
F_i(t)&:=&\sum_{k=-\infty}^{i-1} q^{-4jk}\, \mathbf P_k\(x(t)\ge 0\) = \sum_{k=-\infty}^{i-1} q^{-4jk} \,\mathbf P_0\(x(t)\ge -k\) \nn \\
&=& \sum_{r=-i+1}^{+\infty}\sum_{\ell=r}^{+\infty}  q^{4jr}\, \mathbf P_0\(x(t)= -\ell\) =  \sum_{\ell=-i+1}^{+\infty}\sum_{r=-i+1}^{\ell}  q^{4jr}\, \mathbf P_0\(x(t)= -\ell\) \nn \\
&=& \frac{q^{-4j(i-1)}}{1-q^{4j}} \sum_{\ell=-i+1}^{+\infty}\(1-q^{4j(\ell+i)}\)\mathbf P_0 \(x(t)=\ell\) \nn \\
&=&  \frac{q^{-4j(i-1)}}{1-q^{4j}}  \,\mathbf E_i \[\(1-q^{4jx(t)}\) \, \mathbf 1_{x(t)\ge 1}\]\;. \nn
\end{eqnarray} 
Thus \eqref{1-exp-mom-2} is proved.
\end{proof}

{\br
Since for $q\in (0,1)$
\be
\lim_{t \to \infty} \mathbf E_i \[\(1-q^{4jx(t)}\) \mathbf 1_{x(t)\ge 1}\]=1
\ee
from \eqref{1-exp-mom-2} and \eqref{1-exp-mom-4} we have that
\be\label{limit1}
\lim_{t \to \infty} \mathbb{E}_{\eta^+}\[q^{2J_i(t)}\]=q^{4j \max\{0,i\}} 
\(1 + q^{-4ji}\)
\ee
and 
\be\label{limit2}
\lim_{t \to \infty} \mathbb{E}_{\eta^-}\[q^{2J_i(t)}\]=0
\ee
The limits in \eqref{limit1} and \eqref{limit2} are consistent with a scenario of a shock, respectively, rarefaction fan. Namely,  in the case of shock for a fixed location $i$,  the current $J_i(t)$  in \eqref{limit1} remains bounded as $t \to \infty$ because particles for large times can jump and produce a current only at the location of the moving shock. 
On the contrary, in \eqref{limit2} the current $J_i(t)$ goes to $\infty$ as $t \to \infty$, i.e. the average current $J_i(t)/t$ converges to its stationary value.
\er}

\noindent
It is possible to rewrite \eqref{1-exp-mom-2}, \eqref{1-exp-mom-4} as contour integral. We do this
in the following corollary in order to recover in the case $j=1/2$ the results of \cite{borodin}.

\begin{corollary}
The explicit expression of the $q$-moment  in terms of contour integral reads 
\be
\label{hi1}
\mathbb{E}_{\eta^+}\[q^{2J_k(t)}\]= \frac{q^{4j\max\{0,k\}}}{2\pi i} \: \ointctrclockwise \; e^{- \frac{q^{2j}[2j]_q^3(q^{-1}-q)^2 \, z}{(1+q^{4j}z)(1+z)}\, t} \; \(\frac{1+z}{1+q^{4j}z}\)^k \; \frac{dz}{z}
\ee
where the integration contour includes $0$ and $-q^{-4j}$ but does not include $-1$, 
and
\be
\label{hi2}
\mathbb{E}_{\eta^-}\[q^{2J_k(t)}\]=\frac{q^{-4j\max\{0,k\}}}{2\pi i} \: \ointctrclockwise \; e^{- \frac{q^{-2j}[2j]^3_q(q^{-1}-q)^2 \, z}{(1+q^{-4j}z)(1+z)}\, t} \; \(\frac{1+z}{1+q^{-4j}z}\)^k \; \frac{dz}{z}
\ee
where the integration contour includes $0$ and $-q^{4j}$ but does not include $-1$.
\end{corollary}
\bpr
In order to get \eqref{hi1} and \eqref{hi2} it is sufficient to exploit the contour integral formulation of the modified Bessel function appearing in \eqref{Bessel}, i.e.
\be\label{Bessel0}
I_n(x):=\frac{1}{2\pi i}  \ointctrclockwise e^{(\xi+\xi^{-1})\frac{x}{2}} \xi^{-n-1} \, d\xi
\ee
where the  integration contour includes the origin.
From  \eqref{Bessel} and \eqref{Bessel0} we have
\begin{eqnarray}
\mathbf E_k \[\(1-q^{4jx(t)}\) \mathbf 1_{x(t)\ge 1}\]= 
\sum_{x \ge 1} (1-q^{4jx}) e^{-[4j]_q t} q^{-2j(x-k)} \; I_{x-k}\(2[2j]_q t\) \nn \\
=\frac{q^{2jk}}{2 \pi i} \; e^{-[4j]_q t} \ointctrclockwise  e^{[2j]_q(\xi+\xi^{-1})t} \, \xi^{k-1} 
\sum_{x \ge 1} \, \(1-q^{4jx}\) \, \frac{1}{(\xi q^{2j})^x}\; d\xi \label{serie}
\end{eqnarray}
In order to have the convergence of the series in \eqref{serie} it is necessary to assume $|\xi|\ge q^{-2j}$. Under such assumption we have
\be
\sum_{x \ge 1} \, \(1-q^{4jx}\) \, \frac{1}{(\xi q^{2j})^x}= \frac{\(1-q^{4j}\)\xi}{\(q^{2j}\xi-1\)\(\xi-q^{2j}\)}
\ee
and therefore 
\begin{eqnarray} \label{Contour}
\mathbf E_k \[\(1-q^{4jx(t)}\) \mathbf 1_{x(t)\ge 1}\]= \frac{q^{2jk}}{2 \pi i} \; \ointctrclockwise_{\gamma} f_k(\xi) \, d\xi,
\\ \text{with}\qquad f_k(\xi):= e^{\{[2j]_q(\xi+\xi^{-1})-[4j]_q\}t} \, \frac{\(1-q^{4j}\)  \xi^{k} }{\(q^{2j}\xi-1\)\(\xi-q^{2j}\)} 
\end{eqnarray}
where, from the assumption above, the integration contour $\gamma$ includes 0, $q^{2j}$ and $q^{-2j}$.
From \eqref{1-exp-mom-2}, \eqref{1-exp-mom-4} and \eqref{Contour} we have
\be\label{B+}
\mathbb E_{\eta_\pm}\[q^{2J_k(t)}\]= q^{\pm 4j\max\{0,k\}}\left\{1\pm\frac{q^{\mp 2jk}}{2 \pi i} \; \ointctrclockwise_{\gamma} f_k(\xi)\, d\xi \right\}
\ee
It is easy to verify that $q^{\pm 2j}$  are two simple poles for $f_k(\xi)$ such that 
\be
\text{Res}_{q^{\pm 2j}}(f_k)=\mp q^{\pm 2jk}
\ee
then
\be \label{B+'}
\mathbb E_{\eta_\pm}\[q^{2J_k(t)}\]= \pm q^{\pm 4j\max\{0,k\}}\frac{1}{2 \pi i} \; \ointctrclockwise_{\gamma_\pm} q^{\mp 2jk} f_k(\xi)\, d\xi 
\ee
where $\gamma_\pm$ are now two different contours which include 0 and $q^{\mp 2j}$ and do not include $q^{\pm 2j}$.
In order to get the results in  \eqref{hi1} it is sufficient to perform the change of variable
\be
\xi:= \frac{1+z}{1+q^{4j}z}\, q^{2j}
\ee
to get
\be \label{B+''}
\mathbb{E}_{\eta^+}\[q^{2J_k(t)}\]= - \frac{q^{4j\max\{0,k\}}}{2\pi i} \: \varointclockwise_{\tilde \gamma_+} \; e^{- \frac{q^{2j}[2j]^3_q(q^{-1}-q)^2 \, z}{(1+q^{4j}z)(1+z)}\, t} \; \(\frac{1+z}{1+q^{4j}z}\)^k \; \frac{dz}{z}
\ee
where now the integral is done clockwise over the contour $\tilde \gamma_+$ which includes 0 and $q^{-4j}$ but does not include $-1$. This yields \eqref{hi1} after  changing the integration sense. \eqref{hi2} is obtained similarly from \eqref{B+'} after performing tha change of variables  $\xi:= \frac{1+z}{1+q^{-4j}z}\, q^{-2j}$.
\epr

\begin{remark}
In the case $j=1/2$ formula \eqref{hi1} coincides with the expression in Theorem 1.2 of 
Borodin, Corwin, Sasamoto \cite{borodin} for $n=1$. Indeed defineing
\be
J_k(t)=-N^{BCS}_{k-1}(\eta(t))+ N^{BCS}_{k-1}(\eta(0)), \qquad  N^{BCS}_{k}(\eta):= \sum_{i \le k} \eta_i
\ee
then, if $\eta(0)=\eta_+$ it holds  $J_k(t)=-N^{BCS}_{k-1}(\eta(t))+ 2j \max\{0,k\}$.
As a consequence, from \eqref{hi1}, for $j=1/2$ we have
\be
\label{hi0}
\mathbb{E}_{\eta^+}\[q^{-2N^{BCS}_{k-1}(t)}\]= \frac{1}{2\pi i} \: \ointctrclockwise \; e^{- \frac{(q^{-1}-q)^2 \, z}{(q^{-1}+q z)(1+z)}\, t} \; \(\frac{1+z}{1+q^{2}z}\)^k \; \frac{dz}{z}
\ee
where the integration contour includes 0 and $-q^{-2}$ but does not include -1.
Notice that \eqref{hi0} recovers  the expression in Theorem 1.2 of \cite{borodin} for $\tau=q^{-2}, p=q^{-1}$ 
{(up to a shift $k \to k-1$ which comes from the fact that in $\eta_+$ the first occupied site is 0 in our case while is it choosen to be 1 in \cite{borodin})}.
\end{remark}


\subsection{Product initial condition}

We start with a lemma that is useful in the following.
\bl \label{Lemma:LDP}
Let $x(t)$ be the random walk defined in Lemma \ref{Lemma:N}, $a \in \mathbb R$ and $A \subseteq \mathbb R$ then
\be \label{LDP}
\lim_{t \to \infty} \frac 1 t \log \mathbf E_0 \[a^{x(t)}\, | \; x(t)\in A \]= \sup_{x \in A}\left\{x \log a -{\cal I}(x)\right\} - \inf_{x \in A}{\cal I}(x)
\ee
with
\be \label{I}
{\cal I}(x)=[4j]_q-x + x \log \[q^{2j}\(\frac{x}{2[2j]_q}+ \sqrt{\(\frac{x}{2[2j]_q}\;\)^2+1}\)\]
\ee
\el
\bpr
From large deviations theory \cite{denholla} we know that $x(t)/t$, conditional on
$x(t)/t \in A$, satisfies a large
deviation principle with rate function 
${\cal I}(x) - \inf_{x\in A} {\cal I}(x)$ where ${\cal I}(x)$ is given by
\be
{\cal I}(x):= \sup_{z}\left\{zx- \Lambda(z)\right\}
\ee
with
\be 
\Lambda(z):= \lim_{t \to \infty} \frac 1 t \log \mathbb E\[e^{zx(t)}\]= [2j]_q \(\(e^z-1\)q^{-2j}+\(e^{-z}-1\)q^{2j}\)
\ee
from which it easily follows \eqref{I}. The application
of Varadhan's lemma yields \eqref{LDP}.
\epr

\noindent
We denote by  ${\mathbb E}^{\otimes \mu}$ the expectation of the ASEP$(q,j)$ process on $\mathbb{Z}$  
initialized with the  omogeneous product measure on $\{0,1,\ldots 2j\}^{\mathbb Z}$ with marginals $\mu$
at time 0, i.e. ${\mathbb E}^{\otimes \mu}[f(\eta(t))]= \sum_\eta \left(\otimes_{i\in\mathbb{Z}} \mu(\eta_i)\right) \mathbb E_\eta[f(\eta(t))]$.

\bt[$q$-moment for product initial condition]
\label{current-prod}
Consider a probability measure $\mu$ on $\{0,1,\ldots 2j\}$. Then, 
 for the infinite volume ASEP$(q,j)$, we have
\be \label{Exp-mom-prod}
\mathbb{E}^{\otimes\mu}\[q^{2J_i(t)}\]= \mathbf E_0 \[\(\frac{q^{4j}}{\lambda_q}\)^{x(t)} \mathbf 1_{x(t)\le 0}\]+ 
\mathbf E_0 \[q^{4jx(t)} \(\lambda_{1/q}^{x(t)}-\lambda_{1/q}+ \lambda_q^{-1}\)\mathbf 1_{x(t)\ge 1}\]
\ee
where $\lambda_y:= \sum_{n=0}^{2j} y^n \mu(n)$ 
and $x(t)$ is the random walk defined in Lemma \ref{Lemma:N}.
In particular we have
\be\label{M}
\lim_{t \to\infty} \frac{1}{t} \log\mathbb E^{\otimes\mu}[q^{2J_i(t)}]= \sup_{x \ge 0}\left\{x \log M_q -{\cal I}(x)\right\} -
 \inf_{x \ge 0}{\cal I}(x)
\ee
with $M_q:= \max\{\lambda_q, q^{4j}\lambda_{1/q}\}$ and ${\cal I}(x)$ given by \eqref{I}.
\et
\bpr
From \eqref{Quii} we have
\begin{eqnarray}
&&\mathbb{E}^{\otimes\mu}\[q^{2J_i(t)}\]= \int \otimes\mu(d\eta)\, \mathbb E_\eta\[q^{2J_i(t)}\]
\nn \\
&& = \int \otimes\mu(d\eta) q^{2(N(\eta)-N_i(\eta))} 
+ \sum_{k=-\infty}^{i-1} q^{-4jk} \; \int\otimes\mu(d\eta) \mathbf E_k \left[ q^{4jx(t)}\(q^{-2\eta_{x(t)}}-1\)\, q^{2(N_{x(t)}(\eta)-N_i(\eta))}\right]\;. \nn
\end{eqnarray}
Since
\be
\int \otimes\mu(d \eta) q^{2(N_x(\eta)-N_i(\eta))} = \lambda_q^{i-x} \; \mathbf 1_{\{x \le i\}} + \lambda_{1/q}^{x-i} \; \mathbf 1_{\{x>i\}}
\ee
then, in particular, $\int \otimes\mu(d\eta) q^{2(N(\eta)-N_i(\eta))} = 0$ since $\lambda_q<1$,
where we recall the interpretation of $N(\eta)-N_i(\eta)$ from  lemma \ref{Lemma:N}. 
Hence
\begin{eqnarray}\label{Ct}
\mathbb{E}^{\otimes\mu}\[q^{2J_i(t)}\]&=&  \sum_{k=-\infty}^{i-1} q^{-4jk} \;\sum_{x \in \mathbb Z} \mathbf P_k\(x(t)=x\) \, q^{4jx} \int \otimes\mu(d \eta) \left[q^{2(N_{x+1}(\eta)-N_i(\eta))}-q^{2(N_{x}(\eta)-N_i(\eta))}\right] \nn \\
&=& \(\lambda_q^{-1}-1\) A(t) + \(\lambda_{1/q}-1\)B(t)
\end{eqnarray}
with 
\be
A(t):=\sum_{k\le i-1} q^{-4jk} \sum_{x \le i} \mathbf P_k\(x(t)=x\) q^{4jx} \lambda_q^{i-x}
\ee
and
\be
B(t):=\sum_{k\le i-1} q^{-4jk} \sum_{x \ge i+1} \mathbf P_k\(x(t)=x\) q^{4jx} \lambda_{1/q}^{x-i}
\ee
Now, let $\alpha:= q^{4j}\lambda_q^{-1}$, then
\begin{eqnarray}
A(t)&=& \sum_{k\le i-1} q^{-4jk} \lambda_q^{i}\sum_{x \le i} \mathbf P_k\(x(t)=x\) \alpha^x \nn \\
&=& \sum_{n\ge 1} \lambda_q^n \sum_{m \le n} \mathbf P_0 \(x(t)=m\) \alpha^m \nn \\
&=& \sum_{m \le 0} \alpha^m \mathbf P_0 \(x(t)=m\)  \sum_{n\ge 1} \lambda_q^n + \sum_{m\ge 1} \alpha^m \mathbf P_0\(x(t)=m\)\sum_{n \ge m} \lambda_q^n \nn \\
&=& \frac{1}{1-\lambda_q} \left\{ \lambda_q \, \mathbf E_0 \[\alpha^{x(t)}\, \mathbf 1_{x(t)\le 0}\]+ \mathbf E_0 \[q^{4jx(t)}\, \mathbf 1_{x(t)\ge 1}\]\right\} \label{At}
\end{eqnarray}
Analogously one can prove that
\be
B(t)=\frac{1}{\lambda_{1/q}-1}\left\{\mathbf E_0 \[\beta^{x(t)}\, \mathbf 1_{x(t)\ge 2}\]-\lambda_{1/q}\mathbf E_0 \[q^{4jx(t)}\, \mathbf 1_{x(t)\ge 2}\]\right\}\label{Bt}
\ee
with $\beta=q^{4j}\lambda_{1/q}$ then \eqref{Exp-mom-prod} follows by combining \eqref{Ct}, \eqref{At} and \eqref{Bt}.
\vskip.2cm
\noindent
In order to prove \eqref{M} we use the fact that $x(t)$ has a Skellam distribution with parameters $([2j]_q q^{-2j}t, [2j]_q q^{2j}t)$,
i.e. $x(t)$ is the difference of two independent Poisson random variables with those parameters.  This
implies that
\be
 \mathbf E_0 \[\(\frac{q^{4j}}{\lambda_q}\)^{x(t)} \mathbf 1_{x(t)\le 0}\]=  \mathbf E_0 \[\lambda_q^{x(t)} \mathbf 1_{x(t)\ge 0}\]\;. \nn
\ee
Then we can rewrite \eqref{Exp-mom-prod}
as
\begin{eqnarray} \label{Exp-mom-prod1}
\mathbb{E}^{\otimes\mu}\[q^{2J_i(t)}\]
& = &
\mathbf E_0 \[\(\lambda_q^{x(t)} +\(q^{4j}\lambda_{1/q}\)^{x(t)}\)\mathbf 1_{x(t)\ge 1}\] + \mathbf P_0\(x(t)=0\) \nn \\
&+ & \(\lambda_q^{-1}-\lambda_{1/q}\)\mathbf E_0 \[q^{4jx(t)} \mathbf 1_{x(t)\ge 1}\]\nn \\
& = & \mathbf E_0 \[M_q^{x(t)} \mathbf 1_{x(t)\ge 0}\]  \left(1+\mathcal E_1(t)+\mathcal E_2(t)+\mathcal E_3(t) \right)
\end{eqnarray}
with
\be
\mathcal E_1(t) :=\frac{\mathbf E_0 \[\(\lambda_q^{x(t)} +\(q^{4j}\lambda_{1/q}\)^{x(t)}\)\mathbf 1_{x(t)\ge 1}\]}
{\mathbf E_0 \[M_q^{x(t)} \mathbf 1_{x(t)\ge 0}\]}, \qquad \mathcal E_2(t):=\frac{\mathbf P_0\(x(t)=0\)}{\mathbf E_0 \[M_q^{x(t)} \mathbf 1_{x(t)\ge 0}\]}
\nn
\ee
and
\be
\mathcal E_3(t):=\frac{\(\lambda_q^{-1}-\lambda_{1/q}\)\mathbf E_0 \[q^{4jx(t)} \mathbf 1_{x(t)\ge 1}\]}{\mathbf E_0 \[M_q^{x(t)} \mathbf 1_{x(t)\ge 0}\]}\;. 
\ee
To identify the leaden term in \eqref{Exp-mom-prod1} it remains to prove that, for each $i=1,2,3$ there exists $c_i>0$ such that
\be\label{E}
\sup_{t\ge 0}|\mathcal  E_i(t)| \le c_i
\ee
This would imply, making use of  Lemma \ref{Lemma:LDP}, the result in \eqref{M}. The bound in \eqref{E} is immediate for $i=1,2$. To prove it for $i=3$ it is sufficient to show that there exists $c>0$ such that
\be
\lambda_q^{-1}\mathbf E_0 \[q^{4jx(t)} \mathbf 1_{x(t)\ge 1}\] \le c \, \mathbf E_0 \[\(q^{4j}\lambda_{1/q}\)^{x(t)}\mathbf 1_{x(t)\ge 1}\]\;.
\ee
This follows since there exists $x_*\ge 1$ such that for any $x \ge x_*$ $\lambda_q^{-1}\le \lambda_{1/q}^x$ and then 
\begin{eqnarray}
\lambda_q^{-1}\mathbf E_0 \[q^{4jx(t)} \mathbf 1_{x(t)\ge 1}\]\le \lambda_q^{-1}
\mathbf E_0 \[q^{4jx(t)} \mathbf 1_{1\le x(t)< x_*}\]+ \mathbf E_0 \[q^{4jx(t)} \lambda_{1/q}^{x(t)} \mathbf 1_{x(t)\ge x_*}\]\nn \\
\le \lambda_q^{-1}
\mathbf E_0 \[q^{4jx(t)} \mathbf 1_{1\le x(t)}\]+ \mathbf E_0 \[q^{4jx(t)} \lambda_{1/q}^{x(t)} \mathbf 1_{x(t)\ge 1}\]\nn \\
\le \(1+\lambda_q^{-1}\)\mathbf E_0 \[\(q^{4j} \lambda_{1/q}\)^{x(t)} \mathbf 1_{x(t)\ge 1}\]\;.
\end{eqnarray}
This concludes the proof.
\epr

\noindent
The rest of our paper is devoted to the construction of the process ASEP$(q,j)$ from a quantum spin chain Hamiltonian with
\red{$U_q(\mathfrak{sl}_2)$} symmetry of which we show that it admits a positive ground state. The self-duality functions will then be constructed from application of suitable
symmetries to this ground state and application of proposition \ref{dualprop}.

\section{Algebraic structure and symmetries}
\label{4}

\subsection{The quantum Lie algebra \red{$U_q(\mathfrak{sl}_2)$}}
\label{quantum-algebra}

For $q\in(0,1)$ we consider the algebra with generators $J^{+}, J^{-}, J^{0}$ satisfying the commutation
relations

\begin{eqnarray}
\label{comm-suq2}
[J^+,J^-]=[2J^0]_q, \qquad  [J^0,J^\pm]=\pm J^\pm\;,
\end{eqnarray}
where $[\cdot,\cdot]$ denotes the commutator, i.e. $[A,B] = AB-BA$, and
\begin{equation}
[2J^0]_q :=\frac{q^{2J^0}-q^{-2J^0}}{q-q^{-1}}\;.
\end{equation}
This is the quantum Lie algebra \red{$U_q(\mathfrak{sl}_2)$}, that in the limit $q\to 1$ reduces
to the Lie algebra \red{$\mathfrak{sl}_2$}. Its irreducible representations are  $(2j+1)-$dimensional,
with $j\in\mathbb{N}/2$. They are labeled by the eigenvalues of the
Casimir element
\be
\label{casimir}
C = J^- J^+ + [J^0]_q[J^0+1]_q\;.
\ee
A standard representation {\cite{Lusz}}  of the quantum Lie algebra \red{$U_q(\mathfrak{sl}_2)$}
is given by $(2j+1)\times(2j+1)$ dimensional matrices
defined by
\begin{equation}
\label{stand-repr}
\left\{
\begin{array}{lll}
{J}^+ |n\rangle &=& \sqrt{[2j-\eta]_q [\eta+1]_q}\;| n +1 \rangle
\\
{J}^- |n\rangle &=& \sqrt{[\eta]_q [2j-\eta+1]_q} \;| n-1 \rangle
\\
{J}^0 |n\rangle &=& (\eta-j) \;| n \rangle \;.
\end{array}
\right.
\end{equation}
Here the collection of column vectors $|n\rangle$, with $n \in\{ 0,\ldots, 2j\}$, denote
the standard  orthonormal basis  with respect to the Euclidean scalar product,
i.e.  $|n\rangle = (0,\ldots,0,1,0,\ldots, 0)^T$ with the element $1$ in the $n^{\text{th}}$
position and with the sympol $^{T}$ denoting transposition.
Here and in the following, with abuse of notation, we use the same symbol for a linear
operator and the matrix associated to it in a given basis.
In the representation \eqref{stand-repr}
the ladder operators ${J}^+$ and ${J}^-$
are
one the adjoint of the other, namely
\be\label{transp}
({J}^+)^* = {J}^-
\ee
 and
the Casimir element is given by the diagonal matrix
$$
{C} |n\rangle =[ j]_q[j+1]_q |n\rangle\;.
$$
Later on, in the construction of the q-deformed asymmetric simple exclusion process,
we will consider other representations for which the ladder operators are not adjoint
of each other.
\noindent
For later use, we also observe that the \red{$U_q(\mathfrak{sl}_2)$} commutation relations in \eqref{comm-suq2}
can be rewritten as follows
\begin{eqnarray}
\label{comm-new}
&& q^{J_0} J^+ = q \;  J^+ q^{J_0} \\
&& q^{J_0} J^-= q^{-1}\, J^-q^{J_0} \nonumber\\
&& [J^+,J^-]=[2J^0]_q \nonumber
\end{eqnarray}

\subsection{Co-product structure}
\label{cooooo}

A co-product for the quantum Lie algebra \red{$U_q(\mathfrak{sl}_2)$} is defined as the map
\red{$\Delta: {U_q(\mathfrak{sl}_2)} \to {U_q(\mathfrak{sl}_2)}  \otimes {U_q(\mathfrak{sl}_2)}$}
\begin{eqnarray}
\label{co-product2}
\Delta(J^{\pm}) & = & J^{\pm} \otimes  q^{-J^0} + q^{J^0} \otimes J^{\pm}\;, \nonumber \\
\Delta(J^0) & = & J^0 \otimes 1 +  1\otimes J^0\;.
\end{eqnarray}
The co-product is an isomorphism for the quantum Lie algebra \red{$U_q(\mathfrak{sl}_2)$}, i.e.
\be
\label{coproduct}
[\Delta(J^+),\Delta(J^-)]=[2\Delta(J^0)]_q, \qquad  [\Delta(J^0),\Delta(J^\pm)]=\pm \Delta(J^\pm)\;.
\ee
Moreover it can be easily checked that the co-product satisfies the co-associativity property
\be
\label{co-ass}
(\Delta\otimes 1) \Delta = (1\otimes \Delta) \Delta \;.
\ee
\noindent
Since we are interested in extended systems we will work with the
tensor product over copies of the \red{$U_q(\mathfrak{sl}_2)$} quantum algebra.
We denote by $J_i^{+}, J_i^{-}, J_i^{0}$, with $i\in\mathbb{Z}$, the generators
of the $i^{th}$ copy. Obviously algebra elements of different copies commute.
As a consequence of \eqref{co-ass}, one can define iteratively 
\red{$\Delta^{n}: {U_q(\mathfrak{sl}_2)} \to {U_q(\mathfrak{sl}_2)}^{\otimes (n+1)}$},
i.e. higher power of $\Delta$, as follows: for $n=1$,  from  \eqref{co-product2} we have
\begin{eqnarray}
\label{cooo-product}
\Delta(J_i^{\pm}) & = & J_i^{\pm} \otimes  q^{-J_{i+1}^0} + q^{J_i^0} \otimes J_{i+1}^{\pm} \nonumber \\
\Delta(J_i^0) & = & J_i^0 \otimes 1 +  1 \otimes J_{i+1}^0\;,
\end{eqnarray}
for $n\ge 2$,
\begin{eqnarray}
\label{co-product-L}
\Delta^{n}(J_i^{\pm}) & = & \Delta^{n-1}(J_i^{\pm}) \otimes q^{-J^0_{n+i}}  +  q^{\Delta^{n-1}(J_i^0)} \otimes J_{n+i}^{\pm}  \nonumber\\
\Delta^{n}(J_i^0) & = & \Delta^{n-1}(J_i^0)  \otimes 1 +  \underbrace{1\otimes\ldots\otimes 1}_{n \text{ times}} \otimes {J_{n+i}^0}\;.
\end{eqnarray}

\subsection{The quantum Hamiltonian}
\label{q-h}

Starting from the quantum Lie algebra \red{$U_q(\mathfrak{sl}_2)$} (Section \ref{quantum-algebra})
and the co-product structure (Section \ref{cooooo})   we would like to
construct a linear operator (called ``the quantum Hamiltonian'' in the following and denoted
by $H^{\phantom x}_{(L)}$ for a system of length $L$) with the following properties:
\begin{enumerate}
\item it is \red{$U_q(\mathfrak{sl}_2)$} symmetric, i.e. it admits non-trivial symmetries constructed from
the generators of the quantum algebra; the non-trivial symmetries can then be used
to construct self-duality functions;
\item it can be associated to a continuos time Markov jump process,
i.e. there exists a representation given
by a matrix with non-negative out-of-diagonal elements (which can therefore
be interpreted as the rates of an interacting particle systems) and with zero sum
on each column.
\end{enumerate}
We will approach the first issue in this subsection, whereas the definition
of the related stochastic process is presented in Section \ref{proc}.

\vspace{0.2cm}
\noindent
A natural candidate for the quantum Hamiltonian operator is obtained by applying
the co-product to the Casimir operator $C$ in \eqref{casimir}.
Using the co-product definition \eqref{co-product2},  simple
algebraic manipulations  (cfr. also \cite{B}) yield the following definition.
\bd[Quantum Hamiltonian]
\label{def-qh}
For every $L\in\mathbb{N}$,  $L\ge 2$, we consider the operator $H^{\phantom x}_{(L)}$ defined by
\be
\label{hami}
H^{\phantom x}_{(L)}
:=  \sum_{i=1}^{L-1} H^{i,i+1}_{(L)}
= \sum_{i=1}^{L-1} \left( h^{i,i+1}_{(L)} + c_{(L)} \right) \;,
\ee
where the two-site Hamiltonian is the sum of
\be
\label{const}
c_{(L)} = \frac{(q^{2j}-q^{-2j})(q^{2j+1}-q^{-(2j+1)})}{(q-q^{-1})^2} \underbrace{1\otimes\cdots \otimes 1}_{L \text{ times}}
\ee
and
\be
h^{i,i+1}_{(L)} := \underbrace{1\otimes\cdots \otimes 1}_{(i-1) \text{ times}} \otimes \Delta(C_i) \otimes
\underbrace{1 \otimes \cdots \otimes 1}_{(L-i-1) \text{ times}}
\ee
and, from \eqref{casimir} and \eqref{co-product2},
\be
\Delta(C_i) = \Delta(J_i^-)\Delta(J_i^+) + \Delta([J_i^0]_q) \Delta([J_{i}^0+1]_q)\;.
\ee
Explicitely
\begin{eqnarray}
\label{deltaci}
\Delta(C_i)
& =  &
-  q^{J_i^0}\Bigg \{ J_i^+ \otimes J_{i+1}^-+ J_i^- \otimes J_{i+1}^+
+ \frac {(q^j+q^{-j})(q^{j+1}+q^{-(j+1)})}{2}
[J_i^0]_q \otimes
[J_{i+1}^0]_q \nonumber \\
& &
\quad\qquad + \frac {[j]_q[j+1]_q}{2} \(q^{J_i^0}+q^{-J_i^0}\)\otimes \(q^{J_{i+1}^0}+q^{-J_{i+1}^0}\) \Bigg \} q^{-J_{i+1}^0}
\end{eqnarray}
\ed
\br The diagonal operator $c_{(L)}$ in \eqref{const} has been added so that the ground state $|0\rangle_{(L)} := \otimes_{i=1}^L |0\rangle_i$ is a
right eigenvector with eigenvalue zero,
i.e. $H_{(L)} |0\rangle_{(L)} = 0$ as it is immediately seen using \eqref{stand-repr}.
\er

\bp
\label{h-herm}
In the representation \eqref{stand-repr} the operator $H_{(L)}$ is self-adjoint.
\ep
\bpr
It is enough to consider the non-diagonal part of $H_{(L)}$. Using \eqref{transp} we have
\begin{eqnarray}
&&\(q^{J_i^0} J_i^+ \otimes J_{i+1}^- q^{-J_{i+1}^0} + q^{J_i^0} J_i^- \otimes J_{i+1}^+  q^{-J_{i+1}^0} \)^* \nn \\
&=& J_i^- q^{J_i^0} \otimes  q^{-J_{i+1}^0}  J_{i+1}^++  J_i^+ q^{J_i^0} \otimes q^{-J_{i+1}^0} J_{i+1}^- \nn \\
&=&  q^{J_i^0+1} J_i^- \otimes   J_{i+1}^+ q^{-J_{i+1}^0-1}  +   q^{J_i^0-1}  J_i^+ \otimes J_{i+1}^- q^{-J_{i+1}^0+1}  \nn
\end{eqnarray}
where the last identity follows by using the commutation relations  \eqref{comm-new}. This concludes the proof.
\epr

\subsection{Basic symmetries}
\label{basic}

It is easy to construct symmetries for the operator $H^{\phantom x}_{(L)}$ by using the property that the co-product is
an isomorphism for the \red{$U_q(\mathfrak{sl}_2)$} algebra.

\bt[Symmetries of $H_{(L)}$]
\label{theo-symm}
Recalling \eqref{co-product-L}, we define the operators
\begin{eqnarray}
J_{(L)}^{\pm} & := & \Delta^{L-1}(J_1^{\pm}) = \sum_{i=1}^L q^{J_1^0} \otimes \cdots \otimes q^{J_{i-1}^0} \otimes J_i^{\pm} \otimes q^{-J_{i+1}^0} \otimes \ldots \otimes q^{-J_L^0}\;,
\nonumber \\
J_{(L)}^{0} & := & \Delta^{L-1}(J_1^0) = \sum_{i=1}^L
 \underbrace{1\otimes\cdots \otimes 1}_{(i-1) \text{ times}} \otimes J_i^0 \otimes
\underbrace{1 \otimes \cdots \otimes 1}_{(L-i) \text{ times}}
\;.
\end{eqnarray}
They are symmetries of the Hamiltonian \eqref{hami}, i.e.
\be
[H_{(L)}^{\phantom x},J_{(L)}^{\pm}]= [H_{(L)}^{\phantom x},J_{(L)}^{0}] = 0\;.
\ee
\et
\begin{proof}
We proceed by induction and prove only the result for  $J_{(L)}^{\pm}$ (the case $J^{0}_{(L)}$ is similar).
By construction $J_{(2)}^{\pm} := \Delta(J^{\pm})$ are symmetries
of the two-site Hamiltonian $H^{\phantom x}_{(2)}$. Indeed this is an immediate consequence of the fact that
the co-product defined in \eqref{coproduct} conserves the commutation relations and the Casimir operator
\eqref{casimir} commutes with any other operator in the algebra :
$$
[H^{\phantom x}_{(2)}, J^{\pm}_{(2)}] = [\Delta(C_1),\Delta(J_1^{\pm})] = \Delta ( [C_1,J_1^{\pm}] ) = 0\;.
$$
For the induction step assume now that it holds $[H^{\phantom x}_{(L-1)}, J^{\pm}_{(L-1)}] = 0$. We have
\be
\label{eq-zero}
[H^{\phantom x}_{(L)}, J^\pm_{(L)}]  = [H^{\phantom x}_{(L-1)}, J^\pm_{(L)}]  + [h^{L-1,L}_{(L)}, J^\pm_{(L)}]
\ee
The first term on the right hand side of \eqref{eq-zero}  can be seen to be zero using \eqref{co-product-L} with $i=1$ and
$n=L-1$:
$$
[H^{\phantom x}_{(L-1)}, J^{\pm}_{(L)}] = [H^{\phantom x}_{(L-1)}, J^{\pm}_{(L-1)} q^{-J^0_{L}} + q^{J^0_{(L-1)}} J^\pm_{L}]
$$
Distributing the commutator with the rule $[A,BC] = B[A,C] + [A,B] C$, the induction hypothesis and the fact that spins on
different sites commute imply the claim.
The second term on the right hand side of \eqref{eq-zero} is also seen to be zero by writing
$$
[h^{L-1,L}_{(L)}, J^{\pm}_{(L)}] = [h^{L-1,L}_{(L)}, J^{\pm}_{(L-2)} q^{-\Delta(J^0_{L-1})} + q^{J^0_{(L-2)}} \Delta(J^\pm_{L-1})] = 0\;.
$$
\end{proof}

\br In the case $q=1$, the quantum Hamiltonian in Definition \ref{def-qh} reduces to the (negative of the)
well-known Heisenberg ferromagnetic quantum spin chain with spins $J_i$ satisfying the \red{$\mathfrak{sl}_2$} Lie-algebra.
With abuse of notation for the tensor product, the Heisenberg quantum spin chain reads
\begin{equation}
H^{Heis}_{(L)} = \sum_{i=1}^{L-1} \( J_i^+ J_{i+1}^-+ J_i^- J_{i+1}^+ +2 J^0_i J^0_{i+1}-2j^2 \)\;,
\end{equation}
whose symmetries are given by
$$
J_{(L)}^{\pm,Heis} = \sum_{i=1}^L J_i^{\pm} \qquad \text{and} \qquad J_{(L)}^{0,Heis} = \sum_{i=1}^L J_i^0\;.
$$
\er

\section{Construction of the ASEP$(q,j)$}
\label{proc}

In order to construct a Markov process from the quantum Hamiltonian $H_{(L)}$, we apply item a) of Corollary \ref{corooo} with $A=H_{(L)}$.
At this aim we  need a non-trivial symmetry which yields a non-trivial ground state.
Starting from the basic symmetries of $H_{(L)}$ described in Section
\ref{basic}, and inspired by the analysis of the symmetric case ($q=1$),
it will be convenient to consider the {\em exponential} of those symmetries.

\subsection{The $q$-exponential and its pseudo-factorization}

\bd[$q$-exponential]
We define the $q$-analog of the exponential function  as
\begin{equation}
\label{q-exp}
{\exp}_q(x):= \sum_{n\ge 0} \frac{x^n}{\{n\}_q!}
\end{equation}
where
\begin{equation}
\label{q-num2}
\{n\}_q:= \frac{1-q^n}{1-q}
\end{equation}
\ed
\br The $q$-numbers in \eqref{q-num2} are related to the $q$-numbers in \eqref{q-num} by
the relation $\{n\}_{q^2}= [n]_q q^{n-1}$. This implies  $\{n\}_{q^2} ! = [n]_q! \, q^{n (n-1)/2}$
and therefore
\begin{equation}\label{exptilde}
{\exp}_{q^2}(x)=\sum_{n \ge 0} \frac{x^n}{[n]_q!}\, q^{-n(n-1)/2}
\end{equation}
One could also have defined the $q$-exponential directly in terms of the
q-numbers \eqref{q-num}, namely
\be
\label{q-exp2}
\widetilde{\exp}_q (x) = \sum_{n \ge 0} \frac{x^n}{[n]_q!}
\ee
\er
\vskip.1cm
\noindent
The reason to prefer definition of the $q$-deformed exponential given in \eqref{q-exp},
rather than \eqref{q-exp2}, is that with the first choice we have then a pseudo-factorization property
as described in the following.
\bp[Pseudo-factorization]\label{lemma:E}
\label{pseudo}
Let $\{g_1,\ldots,g_L\}$ and $\{k_1,\ldots,k_L\}$ be  operators  such that for $L \in \mathbb N$
and $g\in\mathbb{R}$
\begin{equation}\label{r}
k_i g_i=r g_ik_i  \qquad\text{for}\quad i=1, \ldots, L\;.
\end{equation}
Define
\begin{equation}
g^{(L)}:=\sum_{i=1}^L  k^{(i-1)} g_i, \quad \text{with} \quad   k^{(i)}:= k_1 \cdot \dots \cdot k_{i} \quad   \text{for $i\ge 1$ and} \quad k^{(0)}=1,
\end{equation}
 then
\begin{equation}\label{fact1}
{\exp}_{r}(g^{(L)})=  {\exp}_{r}(g_1)\cdot {\exp}_{r}(k^{(1)} g_2) \cdot \dots \cdot {\exp}_{r}(k^{(L-1)} g_L)
\end{equation}
Moreover let
\begin{equation}
\hat g^{(L)}:=\sum_{i=1}^L g_i \, h^{(i+1)}, \quad \text{with} \quad   h^{(i)}:= k_i^{-1} \cdot \dots \cdot k^{-1}_{L} \quad   \text{for $i\le L$ and } \quad h^{(L+1)}=1,
\end{equation}
 then
\begin{equation}\label{fact2}
{\exp}_{r}(\hat g^{(L)})=  {\exp}_{r}(g_1 \, h^{(2)}) \cdot \dots \cdot {\exp}_{r}(g_{L-1} \, h^{(L)}) \cdot {\exp}_{r}(g_{L})
\end{equation}
\ep
\vskip.5cm
\noindent
In this section we prove only \eqref{fact1} since the proof of \eqref{fact2} is similar.
We first give a series of Lemma that are useful in the proof.
\begin{lemma}
Let
\begin{equation}
\text{Bin}_r\{n,m\}:= \frac{\{n\}_r!}{\{m\}_r ! \{n-m\}_r!}
\end{equation}
then
\begin{equation}\label{Bin}
r^m {\rm Bin}_r\{n,m\}+  {\rm Bin}_r\{n,m-1\} = {\rm Bin}_r\{n+1, m\}
\end{equation}
\end{lemma}
\begin{proof}
It follows from an immediate computation
\end{proof}
\noindent
\begin{lemma}
For any $n, L \in \mathbb N$, $L\ge 2$
\begin{equation}\label{lem1}
\left(g^{(L)}\right)^n= \sum_{m=0}^n {\rm Bin}\{n,m\}_r \left(g^{(L-1)}\right)^{n-m} (k^{(L-1)}g_{L})^m
\end{equation}
\end{lemma}
\begin{proof}
We prove it by induction on $n$. For $n=1$ it is true because for each $L\ge 2$
\begin{equation}\label{A1}
g^{(L)}=g^{(L-1)}+ k^{(L-1)} g_{L}
\end{equation}
 By \eqref{r}, for any $\ell \in \mathbb N$
\begin{equation}\label{k}
\left(k^{(\ell)}\right)^m g^{(\ell)}= r^m g^{(\ell)} \left(k^{(\ell)}\right)^m
\end{equation}
Suppose that \eqref{lem1} holds for $n$ for any $L \ge 2$, then, using  \eqref{Bin} and \eqref{k} we have
\begin{eqnarray}
\left(g^{(L)}\right)^{n+1}&=&\left(g^{(L-1)}+k^{(L-1)}g_L\right)^{n+1} \nn \\
&=& \sum_{m=0}^n {\rm Bin}_r\{n,m\} \left(g^{(L-1)}\right)^{n-m}\left(k^{(L-1)}g_L\right)^m \cdot \left[g^{(L-1)}+k^{(L-1)}g_L\right] \nn \\
&=&  \sum_{m=1}^n \left[r^m {\rm Bin}_r\{n,m\} +  {\rm Bin}_r\{n,m-1\} \right] \left(g^{(L-1)}\right)^{n+1-m} \left(k^{(L-1)}g_L\right)^m \nn \\ && +
\left(g^{(L-1)}\right)^{n+1} +\left(k^{(L-1)}g_L\right)^{n+1} \nn \\
&=& \sum_{m=0}^{n+1} {\rm Bin}_r\{n+1, m\} \left(g^{(L-1)}\right)^{n+1-m} \left(k^{(L-1)} g_L\right)^{m}
\end{eqnarray}
that proves the lemma.
\end{proof}
\begin{lemma}
For any $n,L \in \mathbb N$, $L\ge 2$ we have
\begin{equation}\label{lem2}
\left(g^{(L)}\right)^n= \{n\}_r! \sum_{m_{L}=0}^n \sum_{m_{L-1}=0}^{n-m_{L}} \dots \sum_{m_2=0}^{n-\sum_{i=3}^{L}m_i} \frac{g_1^{n-\sum_{i=2}^{L}m_i}}{\{n-\sum_{i=2}^{L}m_i\}_r!}  \cdot \prod_{i=2}^{L} \frac{(k^{(i-1)} g_i)^{m_{i}}}{\{m_{i}\}_r!}
\end{equation}
\end{lemma}
\begin{proof}
We  prove it by induction on $L$.  From \eqref{lem1}, for any $n\in \mathbb N$ we have
\begin{equation}
\left(g^{(2)}\right)^n= \left( g_1 + k_1 g_2 \right)^n = \{n\}_r!\sum_{m=0}^n  \frac{\left(g_1\right)^{n-m}}{\{n-m\}_r!} \frac{(k_1g_2)^m}{\{m\}_r!}
\end{equation}
thus \eqref{lem2} is true for $L=2$, $n \in \mathbb N$. Suppose that it holds for $L$ for any $n \in \mathbb N$ then, using \eqref{lem1} we have
\begin{eqnarray}
\left(g^{(L+1)}\right)^n &=&\left(g^{(L)}+k^{(L)}g_{L+1}\right)^n \nn \\
&=& \sum_{m_{L+1}=0}^n {\rm Bin}_r\{n, m_{L+1}\} \left(g^{(L)}\right)^{n-m_{L+1}} \left(k^{(L)}g_{L+1}\right)^{m_{L+1}} \nn \\
&=& \sum_{m_{L+1}=0}^n {\rm Bin}_r\{n, m_{L+1}\}  \bigg(\{n-m_{L+1}\}_r! \sum_{m_{L}=0}^{n-m_{L+1}} \dots \dots \dots \nn \\
&&\hskip-1cm \dots \sum_{m_2=0}^{n-m_{L+1}-\sum_{i=3}^{L}m_i} \frac{g_1^{n-m_{L+1}-\sum_{i=2}^{L}m_i}}{\{n-m_{L+1}-\sum_{i=2}^{L}m_i\}_r!}  \cdot \prod_{i=2}^{L} \frac{(k^{(i-1)} g_i)^{m_{i}}}{\{m_{i}\}_r!} \bigg) \cdot  \left(k^{(L)}g_{L+1}\right)^{m_{L+1}} \nn \\
&=& \{n\}_r! \sum_{m_{L+1}=0}^n \sum_{m_{L}=0}^{n-m_{L+1}} \dots \sum_{m_2=0}^{n-\sum_{i=3}^{L+1}m_i} \frac{g_1^{n-\sum_{i=2}^{L+1}m_i}}{\{n-\sum_{i=2}^{L+1}m_i\}_r!}  \cdot \prod_{i=2}^{L+1} \frac{(k^{(i-1)} g_i)^{m_{i}}}{\{m_{i}\}_r!} \nn
\end{eqnarray}
this proves the lemma.
\end{proof}
\begin{lemma}\label{lemma:10}
Let $L \in \mathbb N$, $L\ge 2$ and for any $i=1, \ldots, L$ let $\mathbf{X}_i\in \mathbb R^{\mathbb N}$ a sequence of real numbers, $\mathbf{X}_i=\{X_i(m)\}_{m\in \mathbb N}$, then
\begin{eqnarray}
\sum_{n=0}^\infty \sum_{m_{L}=0}^n \sum_{m_{L-1}=0}^{n-m_{L}} \dots \sum_{m_2=0}^{n-\sum_{i=3}^{L}m_i} X_1(n-\sum_{i=2}^{L}m_i)  \cdot \prod_{i=2}^{L} X_i(m_i) =     \prod_{i=1}^{L} \; \sum_{m_1=0}^\infty X_i(m_i)
 \label{B'}
\end{eqnarray}
\end{lemma}
\vskip.5cm
\noindent
\bpr
It is sufficient to prove it for $L=2$, the proof of \eqref{B'} follows by an analogous argument.
By performing   the change of variable $n:=m_1+m_2$ we obtain
\begin{eqnarray}
&& \sum_{m_i=0}^\infty \; \prod_{i=1}^2 X_i(m_i)= \sum_{m_1=0}^\infty \sum_{m_2=0}^\infty X_1(m_1)X_2(m_2) \nn \\
&&= \sum_{m_2=0}^\infty \sum_{n=m_2}^\infty X_1(n-m_2)X_2(m_2)= \sum_{n=0}^\infty \sum_{m_2=0}^n X_1(n-m_2)X_2(m_2)\nn
\end{eqnarray}
that yields \eqref{B'} for $L=2$.
\qed

\vskip.5cm
\noindent
{\bf{\small{ P}{\scriptsize ROOF OF PROPOSITION}} \ref{lemma:E}}.
From \eqref{lem2} we have
\begin{eqnarray}
{\exp}_{r}(g^{(L)}) & =&\sum_{n=0}^{\infty} \frac{\left(g^{(L)}\right)^n}{\{n\}_r !}\\
&=& \sum_{n=0}^\infty \sum_{m_{L}=0}^n \sum_{m_{L-1}=0}^{n-m_{L}} \dots \sum_{m_2=0}^{n-\sum_{i=3}^{L}m_i} \frac{g_1^{n-\sum_{i=2}^{L}m_i}}{\{n-\sum_{i=2}^{L}m_i\}_r!}  \cdot \prod_{i=2}^{L} \frac{(k^{(i-1)} g_i)^{m_{i}}}{\{m_{i}\}_r!}  \label{A}\\
&=& {\prod_{i=1}^{L} \, \sum_{m_i=0}^\infty \frac{(k^{(i-1)} g_{i})^{m_{i}}}{\{m_{i}\}_r!}}  \label{B}\\
&=& \prod_{i=1}^{L} \, {\exp}_r(k^{(i-1)}g_{i})
\end{eqnarray}
where the passage from \eqref{A} to \eqref{B} follows from Lemma \ref{lemma:10}.
\epr

\subsection{The exponential symmetry $S^+_{(L)}$}
\label{exps+}

In this Section we identify the symmetry that will be used in the construction of the process ASEP$(q,j)$.
To have a symmetry that has quasi-product form over the sites we preliminary define more convenient
generators of the \red{$U_q(\mathfrak{sl}_2)$} quantum Lie algebra.
Let
\begin{equation}\label{EF2}
 E := q^{J^0} J^+ , \qquad  F:= J^-  q^{-J^0}\qquad \text{and} \qquad  K:= q^{2J^0}
\end{equation}
From the commutation relations \eqref{comm-suq2} we deduce that $(E, F, K)$  verify the relations
\begin{equation}\label{commRel2}
KE= q^2E K \qquad \text{and} \qquad K F = q^{-2} F K  \qquad [E,F] = \frac{K-K^{-1}}{q-q^{-1}}\;.
\end{equation}
Moreover, from Theorem \ref{theo-symm}, the following co-products
\begin{eqnarray}
&&\Delta(E_1):=\Delta(q^{J^0_1}) \cdot \Delta(J^+_1)  = E_1 \otimes  \mathbf 1 + K_1 \otimes   E_2\label{E2}\\
&&\Delta(F_1):= \Delta(J^-_1) \cdot \Delta(q^{-J^0_1})  = F_1 \otimes  K_2^{-1} + \mathbf 1 \otimes   F_2\label{F2}
\end{eqnarray}
are still symmetries of $H_{(2)}$.
In general we can extend  \eqref{E2} and \eqref{F2} to $L$ sites, then we have that
\begin{eqnarray}\label{EL2}
 E^{(L)}&:=& \Delta^{(L-1)} (E_1) \nn \\
&=& \Delta^{(L-1)}(q^{J^0_1}) \cdot \Delta^{(L-1)}(J^+_1) \nn \\
&=& q^{J^0_1} J^+_1 + q^{2J^0_1 + J^0_2} J^+_2 + ...+ q^{2 \sum_{i=1}^{L-1} J^0_i + J^0_L} J^+_L \nn \\
&=&  E_1 +K_1   E_2 + K_1 K_2  E_3 + ...+ K_1 \cdot ... \cdot K_{L-1}  E_L
\end{eqnarray}
\begin{eqnarray}\label{FL234}
 F^{(L)}&:=& \Delta^{(L-1)} (F_1) \nn \\
&=&  \Delta^{(L-1)}(J^-_1)  \cdot \Delta^{(L-1)}(q^{-J^0_1})\nn \\
&=& J^-_1 q^{-J^0_1-2\sum_{i=2}^L J_i^0} + \dots + J^-_{L-1} q^{-J^0_{L-1}-2J^0_L}  + J^-_L q^{-J^0_L}  \nn \\
&=&  F_1 \cdot K_2^{-1} \cdot ... \cdot K_{L}^{-1}+\dots +  F_{L-1}\cdot K_L^{-1}+   F_L
\end{eqnarray}
are symmetries of $H$. If we consider now the symmetry obtained by $q$-exponentiating $E^{(L)}$ then
this operator will pseudo-factorize by Proposition \ref{pseudo}.
\begin{lemma}
\label{lemma:S+}
 The operator
 \begin{equation}
 \label{s+}
 S_{(L)}^+:={\exp}_{q^{2}}( E^{(L)})
\end{equation}
is a  symmetry of $H_{(L)}$.
Its matrix elements are given by
\begin{equation}\label{Selem}
\langle \eta_1, ...,\eta_L |  S^+ | \xi_1, ...,\xi_L \rangle=
\prod_{i=1}^L \sqrt{\binom{\eta_i}{\xi_i}_q \binom{2j-\xi_i}{2j-\eta_i}_q} \cdot \mathbf 1_{\eta_i \ge \xi_i} q^{(\eta_i-\xi_i)\left[1 - j + \xi_i +2\sum_{k=1}^{i-1}(\xi_k-j) \right]}
\end{equation}
\end{lemma}
\vskip.5cm
\noindent
\bpr
From \eqref{commRel2} we know that the operators   $ E_i, K_i$, copies of the operators defined in \eqref{EF2}, verify the conditions \eqref{r} with $r=q^{2}$.
As a consequence, from \eqref{EL2}, \eqref{s+} and Proposition \ref{lemma:E}, we have
\begin{eqnarray}
 S^+_{(L)}
 &=&{\exp}_{q^{2}}( E^{(L)})\nn \\
&=& {\exp}_{q^{2}}( E_1)\cdot {\exp}_{q^{2}}(K_1  E_2) \cdots {\exp}_{q^{2}}(K_1 \cdots K_{L-1}  E_L) \nn\\
&=&{\exp}_{q^{2}} \(q^{J^0_1} J^+_1\) \cdot {\exp}_{q^{2}}\(q^{2J^0_1} q^{J^0_2} J^+_2\) \cdots   {\exp}_{q^{2}}\(q^{2 \sum_{i=1}^{L-1} J^0_i + J^0_L} J^+_L\) \nn \\
&=&  S_1^+  S_2^+ \cdots  S_L^+
\end{eqnarray}
\vskip.1cm
\noindent
where $S^+_i:=  {\exp}_{q^{2}}\(q^{2 \sum_{k=1}^{i-1} J^0_k + J^0_i} J^+_i\)$ has been defined. Using \eqref{exptilde}, we find
\begin{eqnarray}
 S^+_i | \xi_1,\ldots,\xi_L \rangle &=&   \sum_{\ell_i \ge 0} \frac{1}{[\ell_i]_q!}\(q^{2 \sum_{k=1}^{i-1}J_k^0 + J_i^0}J_i^+\)^{\ell_i} q^{-\frac 1 2 \ell_i(\ell_i-1)} | \xi_1,\ldots ,\xi_L\rangle \\
&& \hskip-2.5cm= \sum_{\ell_i \ge 0} \sqrt{\binom{2j-\xi_i}{\ell_i}_q\cdot \binom{\xi_i+\ell_i}{\xi_i}_q} \cdot  q^{\ell_i(1-j+\xi_i)+ 2 \ell_i \sum_{k=1}^{i-1}(\xi_k-j)} |\xi_1,\ldots, \xi_i+\ell_i, \ldots, \xi_L \rangle \nn
\end{eqnarray}
where in the last equality we used \eqref{stand-repr}. Thus we find
\begin{eqnarray}
\label{action-essepiuuu}
 S^+_{(L)} | \xi_1,\ldots,\xi_L \rangle &=&   S^+_1  S_2^+ \dots  S_L^+ | \xi_1,\dots,\xi_L \rangle  \\
&=&  \sum_{\ell_1, \ell_2, \dots, \ell_L \ge 0}\prod_{i=1}^L \Bigg(\sqrt{\binom{2j-\xi_i}{\ell_i}_q\cdot \binom{\xi_i+\ell_i}{\xi_i}_q} \nn \\
&&  \hskip2cm\cdot  \; q^{\ell_i(1-j+\xi_i)+ 2 \ell_i \sum_{k=1}^{i-1}(\xi_k-j)}  \Bigg) |\xi_1+\ell_1,\ldots, \xi_L+\ell_L \rangle  \nn
\end{eqnarray}
form which the matrix elements in \eqref{Selem} are immediately found.
\epr

\subsection{Construction of a positive ground state and  the associated Markov process ASEP$(q,j)$}
\label{sec5.3}

By applying Corollary \ref{corooo} we are now ready to identify the
stochastic process related to the Hamiltonian $H_{(L)}$ in \eqref{hami}.

\noindent
We start from the  state ${\bf |0\rangle} = |0,\ldots,0\rangle$ which
is obviously a trivial ground state of $H_{(L)}$.
We then produce a non-trivial ground state by acting with
the symmetry ${S}^{+}_{(L)}$ in \eqref{s+}, as described
in Remark \ref{nontrivial}.
Using \eqref{action-essepiuuu} we obtain
\begin{eqnarray}\label{g}
|g\rangle &=&  S^+_{(L)} |0, \ldots,0\rangle = \sum_{\ell_1, \ell_2, \dots, \ell_L \ge 0}\prod_{i=1}^L \sqrt{\binom{2j}{\ell_i}_q}
\cdot  \; q^{\ell_i(1+j- 2 j i)} \; |\ell_1, ..., \ell_L \rangle  \nonumber
\end{eqnarray}
Therefore we arrived to a positive ground state (cfr. Remark \ref{nontrivial}).
Following the scheme in Corollary \ref{corooo} we construct the operator $G_{(L)}$ defined by
\be
G_{(L)} |\eta_1,\ldots,\eta_L \rangle =  |\eta_1,\ldots,\eta_L\rangle \langle \eta_1,\ldots,\eta_L | {S}^+ | 0,\ldots,0 \rangle
\ee
In other words $G_{(L)}$ is represented by a diagonal matrix whose coefficients in the standard basis read
\begin{equation}\label{G}
\langle \eta_1,\ldots,\eta_L| G_{(L)} | \xi_1,\ldots,\xi_L \rangle= \prod_{i=1}^L\sqrt{\binom{2j}{\eta_i}_q} \cdot  q^{\eta_i(1+j- 2 j i)} \cdot \delta_{\eta_i=\xi_i}
\end{equation}
Note that $G_{(L)}$ is factorized over the sites, i.e.
\be
\langle \eta_1,\ldots,\eta_L| G_{(L)} | \xi_1,\ldots,\xi_L \rangle= \otimes_{i=1}^L \langle \eta_i| G_i| \xi_i\rangle
\ee
As a consequence of item a) of Corollary \ref{corooo},
the operator ${\cal L}^{(L)}$ conjugated to $H_{(L)}$ via $G^{-1}_{(L)}$, i.e.
\be
\label{tildehami}
{\cal L}^{(L)} = G^{-1}_{(L)} H_{(L)} G_{(L)}
\ee
is the generator of a Markov jump process $\eta(t) = (\eta_1(t),\ldots,\eta_L(t))$
describing particles jumping on the
line $\{1,\ldots,L\}$. The state space of such a process is given by $\{0,\ldots,2j\}^{L}$
and its elements are denoted by  $\eta =  (\eta_1,\ldots,\eta_L)$, where
$\eta_i$  is interpreted as the number of particles at site $i$.
The exclusion rule is due to the fact that on each site can sit no more
than $2j$ particles. The asymmetry is controlled by the parameter
$0 < q \le 1$.
\vskip.4cm

\bp
The action of the Markov generator  ${\cal L}^{(L)}:=  G^{-1}_{(L)} H_{(L)} G_{(L)}$ is given by \eqref{gen}.
\ep
\bpr
From  Proposition \ref{h-herm} we know that $H_{(L)}^*=H_{(L)}$, hence we have that the operator $\tilde{H}_{(L)}:=G_{(L)}H_{(L)}G^{-1}_{(L)}$  is the transposed of the generator $\caL^{(L)}$ defined by \eqref{tildehami}. Then we have to verify that the transition rates to move from $\eta$ to $\xi$ for the Markov process generated by \eqref{gen} are equal to the elements $\langle \xi|\tilde{H}_{(L)}|\eta \rangle$.

\noindent
Since we already know that $\caL^{(L)}$ is a Markov generator, in order to prove the result 
it is sufficient to apply the similarity transformation given by the matrix $G_{(L)}$ defined in \eqref{G}
to the non-diagonal terms of \eqref{deltaci}, i.e. $q^{J^0_{i}} J_i^\pm J_{i+1}^\mp q^{-J^0_{i+1}}$.  We show here the computation only for the first term, being the computation for the other term  similar. \\
 We have
\begin{eqnarray}
& &
\langle \xi_i,\xi_{i+1} | G_i G_{i+1} \cdot q^{J^0_{i}} J_i^+ J_{i+1}^-q^{-J^0_{i+1}} \cdot G^{-1}_i G^{-1}_{i+1}  | \eta_i,\eta_{i+1}\rangle \nn\\
& & =
\langle \xi_i | G_i q^{J^0_{i}} J_i^+ G^{-1}_i  | \eta_i\rangle \otimes
\langle \xi_{i+1} | G_{i+1} J_{i+1}^-q^{-J^0_{i+1}} G^{-1}_{i+1}  | \eta_{i+1}\rangle
\end{eqnarray}
Using \eqref{G} and \eqref{stand-repr} one has
\begin{eqnarray}
\langle \xi_i | G_i q^{J^0_{i}} J_i^+ G^{-1}_i  | \eta_i\rangle  =q^{\eta_i +2 -2ji}\; [2j-\eta_i]_q  \langle \xi_i  | \eta_i +1 \rangle
\end{eqnarray}
and
\begin{eqnarray}
\langle \xi_{i+1} | G_{i+1} J_{i+1}^-q^{-J^0_{i+1}} G^{-1}_{i+1}  | \eta_{i+1}\rangle =
q^{-\eta_{i+1} +2j-1 +2ji}\; [\eta_{i+1}]_q  \langle \xi_{i+1}  | \eta_{i+1} -1 \rangle
\end{eqnarray}
Multiplying the last two expressions one has
\be
\langle \eta^{i+1,i}|  {{\tilde{H}_{(L)}}}|\eta\rangle = q^{\eta_i-\eta_{i+1}+(2j+1)} [2j -\eta_i]_q [\eta_{i+1}]_q
\ee
that corresponds indeed to the rate to move  from $\eta$ to $\eta^{i+1,i}$ in \eqref{gen}. This concludes the proof.
\epr
\vskip.3cm

\br \label{natural}
From item c) of Corollary \ref{corooo}, we have that  the product measure $\mu_{(L)}$ defined by
\be
\mu_{(L)}(\eta)=  \langle \eta|G_{(L)}^2|\eta \rangle
\ee
is a reversible  measure of ${\cal L}^{(L)}$. Notice that it corresponds to the reversible measure $\mathbb P^{(\alpha)}$ defined in \eqref{stat-meas} with the choice $\alpha=1$.

\er

\section{Self-Duality results for the ASEP$(q,j)$}
\label{D}

We now use Proposition \ref{dualprop} and the exponential simmetry  obtained in Section \ref{exps+} to deduce a non-trivial duality function for the ASEP$(q,j)$ process.
We first have  the following remark on  trivial duality functions.

\vskip.3cm
\noindent

\br
From \eqref{Cheap} and item a) of Theorem \ref{basicproptheorem}   it follows that all the functions
\begin{equation}\label{G-function}
d_\alpha(\eta,\xi)= \prod_{i=1}^L  \left(\binom{2j}{\eta_i}_q \cdot  \alpha^n q^{2 \eta_i(1+j- 2 j i)}\right)^{-1} \cdot \delta_{\eta_i=\xi_i}
\end{equation}
are diagonal duality functions for the Markov process with generator ${\cal L}^{(L)}$.
\er


\noindent
We then deduce the main result, i.e. a non-trivial duality function.
\vskip.3cm
\noindent
{\bf{\small{P}{\scriptsize ROOF OF \eqref{dualll} IN THEOREM} \ref{main}}}.
From Proposition \ref{h-herm} we know that $H_{(L)}$ is self-adjoint, then, using Proposition \ref{dualprop} with $A=H_{(L)}$, $G=G_{(L)}$ given by \eqref{G} and $S=S_{(L)}^+$  given by \eqref{Selem} it follows that
\be
G^{-1}_{(L)} S^{+}_{(L)} G^{-1}_{(L)}
\ee
is a self-duality function for the process generated by ${\cal L}^{(L)}$.
 Its elements are computed as follows:
\begin{eqnarray}
&& \langle \eta |G_{(L)}^{-1} S_{(L)}^+ G_{(L)}^{-1}| \xi \rangle =\\
&& = \prod_{i=1}^L\(\sqrt{\binom{2j}{\eta_i}_q} \cdot  q^{\eta_i(1+j- 2 j i)} \)^{-1} \langle \eta |  S^+_i | \xi \rangle  \(\sqrt{\binom{2j}{\xi_i}_q} \cdot  q^{\xi_i(1+j- 2 j i)}\)^{-1}=\\
&& = \prod_{i=1}^L   \sqrt{\binom{\eta_i}{\xi_i}_q \binom{2j-\xi_i}{2j-\eta_i}_q \bigg/  \binom{2j}{\eta_i}_q \binom{2j}{\xi_i}_q }  \cdot q^{(\eta_i-\xi_i)\left[2\sum_{k=1}^{i-1}(\xi_k-j) +\xi_i\right]+ (2 j i -j -1) (\eta_i+\xi_i)} \cdot \mathbf 1_{\xi_i \le \eta_i} =\nonumber \\
&& = q^{\sum_{i=1}^L ((j-1)\eta_i -(3j+1)\xi_i)}\;\prod_{i=1}^L   {\frac{[2j-\xi_i]_q![\eta_i]_q!}{[2j]_q![\eta_i-\xi_i]_q!}}  \cdot \; q^{(\eta_i-\xi_i)\left[2\sum_{k=1}^{i-1}\xi_k +\xi_i\right]+ 4 j i \xi_i} \cdot \mathbf 1_{\xi_i \le \eta_i} \nonumber
\end{eqnarray}
Since both the original process and the dual process conserve the total number of particles it follows that $D_{(L)}$ in \eqref{dualll}
is also a duality function. \qed

\section{A second symmetry and associated self-duality}
\label{C}
Up to now we worked with the symmetry $S_{(L)}^+$ defined in \eqref{s+}.
In this Section we explore other choices for the symmetry and their consequences.

\subsection{Construction of alternative symmetries}

We already observed that the operator $F^{(L)}$ defined in \eqref{FL234} is a symmetry
of $H_{(L)}$. The following Lemma gives the exponential symmetry that is further obtained.
\begin{lemma} \label{lemma:S-'}
 The operator
 \begin{equation}
S_{(L)}^-:={\exp}_{q^{-2}}( F^{(L)})
\end{equation}
is a  symmetry of $H_{(L)}$.
Its matrix elements are given by
\begin{equation}\label{Sel}
\langle \eta_1, ...,\eta_L |  S^-_{(L)} | \xi_1, ...,\xi_L \rangle= \prod_{i=1}^L \sqrt{\binom{\xi_i}{\eta_i}_q \cdot \binom{2j-\eta_i}{2j-\xi_i}_q} \cdot \mathbf 1_{\eta_i \le \xi_i}  q^{-(\xi_i-\eta_i)\left[2\sum_{k=i+1}^L(\eta_k-j)+\eta_i-j+1\right]}
\end{equation}
\end{lemma}
\begin{proof}
From \eqref{commRel2} we know that the operators   $ F_i, {K}_i$, copies of the operator defined in \eqref{EF2}, verify the conditions \eqref{r} with $r=q^{-2}$. Then, from \eqref{FL'} and Proposition \ref{lemma:E}
\begin{eqnarray}
 S^-_{(L)}&=&{\exp}_{q^{-2}}( F^{(L)})\nn \\
&=& {\exp}_{q^{-2}}( F_1 K_2^{-1}\dots K_{L}^{-1})\cdot \dots \cdot {\exp}_{q^{-2}}( F_{L-1}K_L^{-1})\cdot {\exp}_{q^{-2}}( F_L) \nn\\
&=&{\exp}_{q^{-2}} \(J^-_1 q^{-J^0_1-2\sum_{i=2}^L J_i^0}\) \cdot \dots \cdot {\exp}_{q^{-2}}\( J^-_{L-1} q^{-J^0_{L-1}-2J^0_L}\) \cdot {\exp}_{q^{-2}}\(J^-_L q^{-J^0_L}\) \nn \\
&=&  S^-_1   S_2^- \dots  S_L^-
\end{eqnarray}
\vskip.3cm
\noindent
where $ S^-_i:=  {\exp}_{q^{-2}}\(J^-_i q^{-J^0_i-2 \sum_{k=i+1}^{L} J^0_k }\)$. Using \eqref{exptilde} and the fact that $[x]_{q^{-1}}=[x]_q$, we have
\begin{eqnarray}
 S^-_i | \xi_1, ...,\xi_L \rangle & = &  \sum_{\ell_i \ge 0} \frac{1}{[\ell_i]_q!}\(J^-_i q^{-J^0_i-2 \sum_{k=i+1}^{L} J^0_k }\)^{\ell_i} q^{\frac 1 2 \ell_i(\ell_i-1)} | \xi_1, ...,\xi_L\rangle \\
&&\hskip-2.3cm =   \sum_{\ell_i \ge 0}\sqrt{\binom{\xi_i}{\ell_i}_q \cdot \binom{2j-\xi_i+\ell_i}{\ell_i}_q} \; q^{-2\ell_i\sum_{k=i+1}^L (\xi_k-j)} \, q^{\ell_i\(\ell_i-\xi_i+j-1\)}  |\xi_1, ..., \xi_i-\ell_i, ...\xi_L \rangle  \nonumber
\end{eqnarray}
then
\begin{eqnarray}
 S_{(L)}^- | \xi_1, ...,\xi_L \rangle &=&   S^-_1  S_2^- \dots  S_L^- | \xi_1, ...,\xi_L \rangle \nn \\
&=&  \sum_{\ell_1, \ell_2, \dots, \ell_L \ge 0}\prod_{i=1}^L \Bigg(\sqrt{\binom{\xi_i}{\ell_i}_q \cdot \binom{2j-\xi_i+\ell_i}{\ell_i}_q}  \nn \\
&&  \hskip2cm\cdot  q^{-2\ell_i\sum_{k=i+1}^L (\xi_k-\ell_k-j)} \, q^{\ell_i\(\ell_i-\xi_i+j-1\)} \Bigg) |\xi_1-\ell_1, ..., \xi_L-\ell_L \rangle  \nonumber
\end{eqnarray}
From this the matrix elements in \eqref{Sel} immediately follows.
\end{proof}

\noindent
Other symmetries can be obtained as follows. Similarly to Section \ref{exps+}, we consider
\begin{equation}\label{EF'}
\tilde  E :=  J^+  q^{-J^0}, \qquad \tilde F:= q^{J^0} J^- \qquad \text{and} \qquad \tilde K:=q^{2J^0}
\end{equation}
and notice that $(\tilde E, \tilde F,  K)$   (as  $(E, F, K)$ in Section \ref{exps+}) verify the  commutation relations
\begin{equation}\label{commRel'}
\tilde{K}\tilde{E}= q^2\tilde{E} \tilde{K} \qquad \text{and} \qquad \tilde{K} \tilde{F} = q^{-2} \tilde{F} \tilde{K}  \qquad [\tilde{E},\tilde{F}] = \frac{\tilde{K}-\tilde{K}^{-1}}{q-q^{-1}}\;.
\end{equation}
Therefore the following co-products
\begin{eqnarray}
&&\Delta(\tilde E_1):= \Delta(J^+_1) \cdot \Delta(q^{-J^0_1})  =\tilde E_1 \otimes   \tilde{K}_2^{-1} + \mathbf 1\otimes \tilde  E_2\label{E'}\\
&&\Delta(\tilde F_1):=\Delta(q^{J^0_1}) \cdot \Delta(J^-_1)  =\tilde F_1 \otimes  \mathbf 1 + \tilde{K}_1 \otimes \tilde F_2 \label{F}
\end{eqnarray}
are  symmetries of $H_{(2)}$.
In general we can extend  \eqref{E'} and \eqref{F} to $L$ sites, then we have that
\begin{eqnarray}\label{EL'}
\tilde E^{(L)}&:=& \Delta^{(L-1)}\tilde E_1 \nn \\
&=& \Delta^{(L-1)}(J^+_1) \cdot  \Delta^{(L-1)}(q^{-J^0_1})\nn \\
&=& J^+_1 q^{-J^0_1-2\sum_{i=2}^L J_i^0} + \dots + J^+_{L-1} q^{-J^0_{L-1}-2J^0_L}  + J^+_L q^{-J^0_L}  \nn \\
&=&\tilde  E_1 \cdot \tilde{K}_2^{-1} \cdot ... \cdot \tilde{K}_{L}^{-1}+\dots + \tilde E_{L-1}\cdot \tilde{K}_L^{-1}+  \tilde E_L
\end{eqnarray}
\begin{eqnarray}\label{FL'}
\tilde F^{(L)}&:=& \Delta^{(L-1)}\tilde F_1 \nn \\
&=& \Delta^{(L-1)}(q^{J^0_1}) \cdot \Delta^{(L-1)}(J^-_1) \nn \\
&=& q^{J^0_1} J^-_1 + q^{2J^0_1 + J^0_2} J^-_2 + ...+ q^{2 \sum_{i=1}^{L-1} J^0_i + J^0_L} J^-_L \nn \\
&=& \tilde F_1 +\tilde{K}_1  \tilde F_2 + \tilde{K}_1 \tilde{K}_2 \tilde F_3 + ...+ \tilde{K}_1 \cdot ... \cdot \tilde{K}_{L-1} \tilde F_L
\end{eqnarray}
are symmetries of $H_{(L)}$.

\vskip.3cm
\br Notice that $\tilde E_{(L)}$ (respectively $\tilde F_{(L)}$) is related to $F_{(L)}$ (respectively $E_{(L)}$) by a transposition.
More precisely, using \eqref{comm-new}, one has
\begin{eqnarray}\label{FE}
(\tilde E^{(L)})^* &= &  q^{-J^0_1}J^-_1 q^{-2\sum_{i=2}^L J_i^0} + \dots + q^{-J^0_{L-1}}J^-_{L-1} q^{-2J^0_L}  + q^{-J^0_L}J^-_L  \nn \\
&= & q\left( J^-_1q^{-J^0_1} q^{-2\sum_{i=2}^L J_i^0} + \dots + J^-_{L-1}q^{-J^0_{L-1}} q^{-2J^0_L}  + J^-_Lq^{-J^0_L} \right)\nn \\
&=&q F^{(L)}
\end{eqnarray}
\begin{eqnarray}\label{FEZZ}
(\tilde F^{(L)})^* &= & J^+_1 q^{J^0_1}+ q^{2J^0_1}  J^+_2 q^{J^0_2}+ ...+ q^{2 \sum_{i=1}^{L-1} J^0_i} J^+_L q^{J^0_L} \nn \\
&= & q^{-1}\left( q^{J^0_1} J^+_1 + q^{2J^0_1 + J^0_2} J^+_2 + ...+ q^{2 \sum_{i=1}^{L-1} J^0_i + J^0_L} J^+_L\right)\nn \\
&=&q^{-1} E^{(L)}
\end{eqnarray}
\er


\vskip1cm
\noindent
By exponentiating $\tilde{E}_{(L)}$ and $\tilde{F}_{(L)}$ the following two symmetries $\tilde{S}_{(L)}^+$ and $\tilde{S}_{(L)}^-$ are obtained.

\begin{lemma} \label{lemma:S+'}
 The operator
 \begin{equation}
 \label{essepiu}
\tilde S_{(L)}^+:={\exp}_{q^{2}}(\tilde E^{(L)})
\end{equation}
is a  symmetry of $H_{(L)}$.
Its matrix elements are given by
\begin{equation}\label{Sel22}
\langle \eta_1, ...,\eta_L | \tilde S_{(L)}^+ | \xi_1, ...,\xi_L \rangle= \prod_{i=1}^L \sqrt{\binom{2j-\xi_i}{2j-\eta_i}_q \cdot \binom{\eta_i}{\xi_i}_q} \, q^{-(\eta_i-\xi_i)\left[2\sum_{k=i+1}^L(\eta_k-j)+\eta_i -j -1\right]} \; \cdot \mathbf 1_{\xi_i \le \eta_i}
\end{equation}
\end{lemma}
\begin{proof}
From \eqref{commRel'} we know that the operators   $\tilde E_i, \tilde{K}_i$, copies of the operators defined in \eqref{EF'}, verify the conditions \eqref{r} with $r=q^{2}$. Then, from \eqref{EL'} and Proposition \ref{lemma:E}
\begin{eqnarray}
\tilde S_{(L)}^+&=&{\exp}_{q^{2}}(\tilde E^{(L)})\nn \\
&=& {\exp}_{q^{2}}(\tilde E_1 \tilde K_2^{-1}\dots \tilde K_{L}^{-1})\cdot \dots \cdot {\exp}_{q^{2}}(\tilde E_{L-1}\tilde K_L^{-1})\cdot {\exp}_{q^{2}}(\tilde E_L) \nn\\
&=&{\exp}_{q^{2}} \(J^+_1 q^{-J^0_1-2\sum_{i=2}^L J_i^0}\) \cdot \dots \cdot {\exp}_{q^{2}}\( J^+_{L-1} q^{-J^0_{L-1}-2J^0_L}\) \cdot {\exp}_{q^{2}}\(J^+_L q^{-J^0_L}\) \nn \\
&=&\tilde  S_1^+\tilde  S_2^+ \dots \tilde S_L^+
\end{eqnarray}
\vskip.3cm
\noindent
where $\tilde S^+_i:=  {\exp}_{q^{2}}\(J^+_i q^{-J^0_i-2 \sum_{k=i+1}^{L} J^0_k }\)$. Using \eqref{exptilde}, we have
\begin{eqnarray}
\tilde S^+_i | \xi_1, ...,\xi_L \rangle &=&   \sum_{\ell_i \ge 0} \frac{1}{[\ell_i]_q!}\(J^+_i q^{-J^0_i-2 \sum_{k=i+1}^{L} J^0_k }\)^{\ell_i} q^{-\frac 1 2 \ell_i(\ell_i-1)} | \xi_1, ...,\xi_L\rangle \\
&& \hskip-2.4cm = \sum_{\ell_i \ge 0} \sqrt{\binom{2j-\xi_i}{\ell_i}_q \cdot \binom{\xi_i+\ell_i}{\xi_i}_q} q^{-2\ell_i \sum_{k=i+1}^L(\xi_k-j)} \, q^{-\ell_i \(\xi_i+\ell_i-j-1\)}|\xi_1, ..., \xi_i+\ell_i, ...\xi_L \rangle \nn
\end{eqnarray}
then
\begin{eqnarray}
\tilde S_{(L)}^+ | \xi_1, ...,\xi_L \rangle &=&  \tilde S^+_1 \tilde S_2^+ \dots \tilde S_L^+ | \xi_1, ...,\xi_L \rangle \nn \\
&=&  \sum_{\ell_1, \ell_2, \dots, \ell_L \ge 0}\prod_{i=1}^L \Bigg( \sqrt{\binom{2j-\xi_i}{\ell_i}_q \cdot \binom{\xi_i+\ell_i}{\xi_i}_q} \nn \\
&&  \hskip1cm\cdot  \;  q^{-2\ell_i \sum_{k=i+1}^L(\xi_k+\ell_k-j)} \, q^{-\ell_i \(\xi_i+\ell_i-j-1\)} \Bigg) |\xi_1+\ell_1, ..., \xi_L+\ell_L \rangle  \nonumber
\end{eqnarray}
Hence   the matrix elements of $\tilde S^+_{(L)}$ are given by \eqref{Sel22}.
\end{proof}

\vskip1cm

\begin{lemma} \label{lemma:S-}
 The operator
 \begin{equation}
\tilde S_{(L)}^-:={\exp}_{q^{-2}}(\tilde F^{(L)})
\end{equation}
is a  symmetry of $H_{(L)}$.
Its matrix elements are given by
\begin{equation}\label{Sel33}
\langle \eta_1, ...,\eta_L | \tilde S_{(L)}^- | \xi_1, ...,\xi_L \rangle=  \prod_{i=1}^L \sqrt{\binom{\xi_i}{\eta_i}_q \binom{2j-\eta_i}{2j-\xi_i}_q}\cdot  q^{(\xi_i-\eta_i)\left[2\sum_{k=1}^{i-1}(\xi_k-j) -\xi_i +1 +j \right]}  \cdot \mathbf 1_{\eta_i \le \xi_i}
\end{equation}
\end{lemma}
\begin{proof}
From \eqref{commRel'} we know that the operators   $\tilde F_i, \tilde{K}_i$, copies of the operators  defined in \eqref{EF'}, verify the conditions \eqref{r} with $r=q^{-2}$. Then, from \eqref{EL'} and Proposition \ref{lemma:E}
\begin{eqnarray}
\tilde S_{(L)}^-&=&{\exp}_{q^{-2}}(\tilde F^{(L)})\nn \\
&=& {\exp}_{q^{-2}}(\tilde F_1)\cdot {\exp}_{q^{-2}}(\tilde K_1 \tilde F_2) \cdot \dots \cdot {\exp}_{q^{-2}}(\tilde K_1 \cdot \dots \cdot \tilde K_{L-1} \tilde F_L) \nn\\
&=&{\exp}_{q^{-2}} \(q^{J^0_1} J^-_1\) \cdot {\exp}_{q^{-2}}\(q^{2J^0_1} q^{J^0_2} J^-_2\) \cdot   \dots\cdot {\exp}_{q^{-2}}\(q^{2 \sum_{i=1}^{L-1} J^0_i + J^0_L} J^-_L\) \nn \\
&=&\tilde  S^-_1 \tilde  S_2^- \dots \tilde S_L^-
\end{eqnarray}
\vskip.3cm
\noindent
where $\tilde S^-_i:=  {\exp}_{q^{-2}}\(q^{2 \sum_{k=1}^{i-1} J^0_k + J^0_i} J^-_i\)$. Using \eqref{exptilde} and the fact that $[x]_{q^{-1}}=[x]_q$, we have
\begin{eqnarray}
\tilde S^-_i | \xi_1, ...,\xi_L \rangle &=&   \sum_{\ell_i \ge 0} \frac{1}{[\ell_i]_q!}\(q^{2 \sum_{k=1}^{i-1}J_k^0 + J_i^0}J_i^-\)^{\ell_i} q^{\frac 1 2 \ell_i(\ell_i-1)} | \xi_1, ...,\xi_L\rangle \\
&& \hskip-2.4cm  =  \sum_{\ell_i \ge 0}  \sqrt{\binom{2j-\xi_i+\ell_i}{\ell_i}_q\cdot \binom{\xi_i}{\ell_i}_q} \cdot  q^{\ell_i(1+j-\xi_i)+ 2 \ell_i \sum_{k=1}^{i-1}(\xi_k-j)} |\xi_1, ..., \xi_i-\ell_i, ...\xi_L \rangle  \nonumber
\end{eqnarray}
then
\begin{eqnarray}
\tilde S_{(L)}^- | \xi_1, ...,\xi_L \rangle &=&  \tilde S^-_1 \tilde S_2^- \dots \tilde S_L^- | \xi_1, ...,\xi_L \rangle \nn \\
&=&  \sum_{\ell_1, \ell_2, \dots, \ell_L \ge 0}\prod_{i=1}^L \Bigg(\sqrt{\binom{2j-\xi_i+\ell_i}{\ell_i}_q\cdot \binom{\xi_i}{\ell_i}_q} \nn \\
&&  \hskip2cm\cdot  \;  q^{\ell_i(1+j-\xi_i)+ 2 \ell_i \sum_{k=1}^{i-1}(\xi_k-j)}\Bigg) |\xi_1-\ell_1, ..., \xi_L-\ell_L \rangle  \nonumber
\end{eqnarray}
Hence the matrix elements of $\tilde S^-_{(L)}$ are given by \eqref{Sel33}.
\end{proof}

\noindent
As it was done with the ground state $S^+_{(L)}|0,\ldots,0\rangle$, one could wonder what Markov process
is obtained if one uses the ground state $\tilde S^+_{(L)}|0,\ldots,0\rangle$.
One can check by an explicit computation (not reported here) that if $H_{(L)}$ is transformed by a similarity transformation
$\tilde G_{(L)}$ given by
\be
\tilde G_{(L)} |\eta_1,\ldots,\eta_L\rangle = |\eta_1,\ldots,\eta_L\rangle \langle \eta_1,\ldots,\eta_L |\tilde S^+_{(L)}|0,\ldots,0\rangle
\ee
one recovers the ASEP$(q,j)$ Markov jump process.

\subsection{Construction of alternative self-duality functions}

One can wonder what other dualities are found using the other symmetries of the previous Section.
Using $S^-_{(L)}$ one finds a duality function which is the transpose of \eqref{dualll}.
In the same way $\tilde S^+_{(L)}$ and $\tilde S^-_{(L)}$ give duality functions that are related by a transposition.
Such duality function is different from \eqref{dualll} and is given by \eqref{duality22} that we are going to prove below.

\vskip.4cm

\noindent
{\bf{\small{P}{\scriptsize ROOF OF \eqref{duality22} IN THEOREM} \ref{main}}}.
From Proposition \ref{h-herm} we know that $H_{(L)}$ is self-adjoint, then, using Proposition \ref{dualprop} with $A=H_{(L)}$, $G=G_{(L)}$ given by \eqref{G} and $S=\tilde S_{(L)}^-$  given by \eqref{Selem} it follows that
\be
G^{-1}_{(L)} \tilde S^{-}_{(L)} G^{-1}_{(L)}
\ee
is a self-duality function for the process generated by ${\cal L}^{(L)}$.
 Its elements are computed as follows:
\begin{eqnarray}
&& \langle \eta |G_{(L)}^{-1} \tilde S_{(L)}^- G_{(L)}^{-1}| \xi \rangle =\\
&& = \prod_{i=1}^L\(\sqrt{\binom{2j}{\eta_i}_q} \cdot  q^{\eta_i(1+j- 2 j i)} \)^{-1} \langle \eta | \tilde S^-_i | \xi \rangle  \(\sqrt{\binom{2j}{\xi_i}_q} \cdot  q^{\xi_i(1+j- 2 j i)}\)^{-1}=\\
&& = \prod_{i=1}^L   \sqrt{\binom{\xi_i}{\eta_i}_q \binom{2j-\eta_i}{2j-\xi_i}_q \bigg/  \binom{2j}{\eta_i}_q \binom{2j}{\xi_i}_q }  \cdot
q^{(\xi_i-\eta_i)\left[2\sum_{k=1}^{i-1}(\xi_k-j) -\xi_i\right]+ (2 j i-j-1) (\eta_i+\xi_i)} \cdot \mathbf 1_{\eta_i \le \xi_i} =\nonumber \\
&& = q^{\sum_{i=1}^L ((j-1)\xi_i -(3j+1)\eta_i)}\;\prod_{i=1}^L   \frac{[2j-\eta_i]_q![\xi_i]_q!}{[2j]_q![\xi_i-\eta_i]_q!}  \cdot \; q^{(\xi_i-\eta_i)\left[2\sum_{k=1}^{i-1}\xi_k -\xi_i\right]+ 4 j i \eta_i} \cdot \mathbf 1_{\eta_i \le \xi_i} \nonumber
\end{eqnarray}
Since both the original process and the dual process conserve the total number of particles it follows that $D_{(L)}'$ in \eqref{duality22}
is also a duality function. \qed


\subsection{Comparison with the Sch\"{u}tz duality in the case $j=1/2$.}

Consider the duality matrix $D'$ computed in \eqref{duality22}, then the associated duality function is
\begin{equation*}
D'_{(L)}(\eta,\xi)= \prod_{i=1}^L   \frac{\binom{\eta_i}{\xi_i}_q}{\binom{2j}{\xi_i}_q}\, \cdot \; q^{(\eta_i-\xi_i)\left[2\sum_{k=1}^{i-1}\eta_k -\eta_i\right]+ 4 j i \xi_i} \cdot \mathbf 1_{\xi_i \le \eta_i}
\end{equation*}
For $j=1/2$ both $\xi_i$ and $\eta_i$ take values in $\{0,1\}$ then
\begin{equation} \label{simp}
\eta^2_i \equiv \eta_i \qquad \qquad \text{and for $\xi_i\le \eta_i$,} \qquad \xi_i \eta_i \equiv \xi_i
\end{equation}
hence, assuming that $\xi_i\le \eta_i$ for all $i$, we have
\begin{equation*}
\sum_{i=1}^L(\eta_i-\xi_i)\eta_i= \sum_{i=1}^L\eta_i^2- \sum_{i=1}^L\xi_i\eta_i= \sum_{i=1}^L\eta_i- \sum_{i=1}^L\xi_i= N-M
\end{equation*}
where $N$ and $M$ are the total numbers of particles respectively in the configurations $\eta$ and $\xi$. Thus
\begin{equation*}
\prod_{i=1}^L  \, \frac{\binom{\eta_i}{\xi_i}_q}{\binom{2j}{\xi_i}_q}\, \cdot  \, q^{-(\eta_i-\xi_i)\eta_i} \cdot \mathbf 1_{\xi_i \le \eta_i}= q^{-\sum_{i=1}^L(\eta_i-\xi_i)\eta_i} \, \cdot \prod_{i=1}^L   \mathbf 1_{\xi_i \le \eta_i}= c \, \cdot  \mathbf 1_{\{\xi_i \le \eta_i, \, \forall i\}}
\end{equation*}
On the other hand, assuming that $\xi_i\le \eta_i$,
 we have
\begin{equation*}
\eta_i-\xi_i= \mathbf 1_{\eta_i=1,\xi_1=0}, \qquad \text{then} \qquad \prod_{i=1}^L   \; q^{2(\eta_i-\xi_i) \sum_{k=1}^{i-1}\eta_k}= \prod_{i:\eta_i=1, \xi_i=0}  \; q^{2\sum_{k=1}^{i-1}\eta_k}
\end{equation*}
then, for $j=1/2$
\begin{equation*}
D'(\eta,\xi)= c \,\cdot  \mathbf 1_{\{\xi_i \le \eta_i, \, \forall i\}} \, \cdot q^{2\sum_{i=1}^L i \xi_i}  \;  \prod_{{i:\eta_i=1, \xi_i=0}}  \; q^{2 \sum_{k=1}^{i-1}\eta_k}
\end{equation*}
Now, using the Sch\"{u}tz notation, one may  represent a given $M$-particles configuaration by the set $C$ of occupied sites. More precisely, let $M$ be the total number of the configuartion $\xi$, we denote by $C:=\{k_1, \dots, k_M\}$ the set of occupies sites $k_i \in \{1, \dots, L\}$ $k_i \le k_{i+1}$. With this notation we have
\begin{equation*}
\sum_{i=1}^L i \xi_i= \sum_{m=1}^M k_m
\end{equation*}
On the other hand, for the configuration $\eta$ we denote by $N_i$, $i=1, \dots, L$ the  number of particles at the left of $i$ (with site $i$ included):
\begin{equation*}
N_i:=\sum_{k=1}^i \eta_k
\end{equation*}
\noindent
With this notation we have
\begin{equation}\label{dual1/2}
D'_{(L)}(\eta,\xi)= c \,\cdot  \mathbf 1_{\{\xi_i \le \eta_i, \, \forall i\}} \, \cdot q^{2\sum_{m=1}^M k_m}  \;   \; q^{2 \sum_{i:\eta_i=1, \xi_i=0}  N_{i-1}}
\end{equation}
Now, assuming that  $\xi_i\le \eta_i$ for all $i$, we have
\begin{equation}\label{N0}
\sum_{i:\eta_i=1, \xi_i=0}  N_{i-1}= \sum_{i:\eta_i=1}  N_{i-1}-\sum_{i:\eta_i=1, \xi_i=1}  N_{i-1}
\end{equation}

\noindent
Let now $N$ be the total number of particles in the configuration $\eta$, then we  prove that
\begin{equation}\label{N1}
\sum_{i:\eta_i=1}  N_{i-1}=\frac{N(N-1)}{2}
\end{equation}
We have
\begin{equation*}
\sum_{i:\eta_i=1}  N_{i-1}=\sum_{i:\eta_i=1}  \eta_i N_{i-1}= \sum_{i:\eta_i=1}  \sum_{k=1}^{i-1} \eta_i \eta_k
\end{equation*}
On the other hand
\begin{eqnarray*}
N^2&=& \left(\sum_{i=1}^L \eta_i\right)^2= \sum_{i=1}^L \sum_{k=1}^L \eta_i \eta_k \nn \\
&=&  \sum_{i=1}^L \sum_{k=1}^{i-1} \eta_i \eta_k + \sum_{i=1}^L \eta_i^2 +  \sum_{i=1}^L \sum_{k=i+1}^{L} \eta_i \eta_k \nn \\
&=&  2\sum_{i=1}^L \sum_{k=1}^{i-1} \eta_i \eta_k + N
\end{eqnarray*}
where the last identity follows because
\begin{equation*}
\sum_{i=1}^L \sum_{k=1}^{i-1} \eta_i \eta_k=\sum_{i=1}^L \sum_{k=i+1}^{L} \eta_i \eta_k \nn
\end{equation*}
and since , from the left identity in \eqref{simp},
\begin{equation*}
 \sum_{i=1}^L \eta_i^2 = \sum_{i=1}^L \eta_i=N
\end{equation*}
then \eqref{N1} is proved. On the other hand, from the right identity in \eqref{simp} we have
\begin{eqnarray}\label{N2}
\sum_{i:\eta_i=1, \xi_i=1}  N_{i-1} &=& \sum_{i=1}^{L}\eta_i \xi_i \sum_{k=1}^{i-1} \eta_k \nn \\
&=&  \sum_{i=1}^{L}\xi_i \sum_{k=1}^{i-1} \eta_k \nn \\
&=& \sum_{m=1}^M \sum_{k=1}^{k_m-1}\eta_k \nn \\
&=& \sum_{m=1}^M N_{k_m-1}
\end{eqnarray}
Finally from \eqref{N0}, \eqref{N1} and \eqref{N2} we have
\begin{equation}\label{N3}
\sum_{i:\eta_i=1, \xi_i=0}  N_{i-1}=\frac{N(N-1)}{2}-\sum_{m=1}^M N_{k_m-1}
\end{equation}
Finally we have that $\xi_i \le \eta_i$ for all $i$ if and only if all the sites $\{k_1, \dots, k_M\}$ are occupied sites for the configuration $\eta$, then from \eqref{dual1/2} and \eqref{N3} we have
\begin{eqnarray}\label{Dual1/2}
D'(_{(L)}\eta,\xi)&=& c' \,\cdot  \mathbf 1_{\{\xi_i \le \eta_i, \, \forall i\}} \, \cdot q^{2\sum_{m=1}^M k_m}  \;   \; q^{-2\sum_{m=1}^M N_{k_m-1}} \nn \\
&=& c' \,\cdot \prod_{m=1}^M  \, q^{2 k_m}  \;   \; q^{-2N_{k_m-1}} \, \cdot \, \eta_{k_m}
\end{eqnarray}
that is the Sch\"{u}tz self-duality function (up to a sign, i.e. $q^{2k_m}$ instead of $q^{-2 k_m}$).

%
%
%
%
%
%
%

\vspace{0.5cm}

{\bf Acknowledgments.} The authors thank Eric Koelink 
for several useful discussions in the initial stage of this work. 
The research of C. Giardin\`a and G. Carinci has been partially 
supported by FIRB 2010 (grant n. RBFR10N90W). 
{Furthermore G. Carinci acknowledges financial support 
by National Group of Mathematical Physics (GNFM-INdAM).}
T. Sasamoto is grateful for the support from KAKENHI 22740054 and Sumitomo Foundation.
We thank the Galileo Galilei Institute for Theoretical Physics for the hospitality and the
INFN for partial support during the completion of this work.

\end{document}